\documentclass{article}
\usepackage[english]{babel}
\usepackage[utf8]{inputenc}
\usepackage[T1]{fontenc}
\usepackage{graphicx}
\usepackage{amsmath,amssymb,amsfonts,amsthm}
\usepackage{geometry}
\usepackage{mathrsfs}
\usepackage{nicefrac}
\usepackage{booktabs}
\usepackage{multirow}
\usepackage{enumitem}
\usepackage{float}
\usepackage{subcaption}
\usepackage{diagbox}
\usepackage{hyperref}
\usepackage{authblk}
\usepackage{bm}
\usepackage{upgreek}
\hypersetup{colorlinks = true, linkcolor=red, citecolor=blue}

\newcommand{\psh}[2]{\langle #1 \, , #2 \rangle}

\newcommand{\sumlim}[2]{\sum\limits_{#1}^{#2}}
\newcommand{\itg}[4]{\int_{#1}^{#2} #3 \, \mathrm{d}#4}

\newcommand{\R}{\mathbb R}
\newcommand{\N}{\mathbb N}
\newcommand{\Z}{\mathbb Z}

\newcommand{\br}{{\bf r}}

\newcommand{\cK}{{\mathcal K}}

\newcommand{\cO}{{\mathcal O}}

\newcommand{\cE}{{\mathcal E}}
\newcommand{\cR}{{\mathcal R}}

\providecommand{\keywords}[1]
{
  \small	
  \textbf{Keywords.} #1
}

\title{Variational projector-augmented wave method: a full-potential approach for electronic structure calculations in solid-state physics}
\author{Mi-Song Dupuy\thanks{Zentrum Mathematik, Technische Universit\"at M\"unchen, Boltzmannstra\ss e 3, 85747 Garching, Germany \\
(Email: \url{dupuy@ma.tum.de})}}

\begin{document}
\numberwithin{equation}{section}
\newtheorem{theo}{Theorem}[section]
\newtheorem{prop}[theo]{Proposition}
\newtheorem{note}[theo]{Remark}
\newtheorem{lem}[theo]{Lemma}
\newtheorem{cor}[theo]{Corollary}
\newtheorem{definition}[theo]{Definition}
\newtheorem{assumption}{Assumption}

\maketitle

\begin{abstract}
\begin{small}
In solid-state physics, energies of crystals are usually computed with a plane-wave discretization of Kohn-Sham equations. 
However the presence of Coulomb singularities requires the use of large plane-wave cut-offs to produce accurate numerical results. 
In this paper, an analysis of the plane-wave convergence of the eigenvalues of periodic linear Hamiltonians with Coulomb potentials using the variational projector-augmented wave (VPAW) method is presented. 
In the VPAW method, an invertible transformation is applied to the original eigenvalue problem, acting locally in balls centered at the singularities. In this setting, a generalized eigenvalue problem needs to be solved using plane-waves.
We show that cusps of the eigenfunctions of the VPAW eigenvalue problem at the positions of the nuclei are significantly reduced. 
These eigenfunctions have however a higher-order derivative discontinuity at the spheres centered at the nuclei.
By balancing both sources of error, we show that the VPAW method can drastically improve the plane-wave convergence of the eigenvalues with a minor additional computational cost. 
Numerical tests are provided confirming the efficiency of the method to treat Coulomb singularities. 
\end{small}
\end{abstract}


\section*{Introduction}

In solid-state physics, the electronic properties of a material are given by the bottom of the spectrum of a Kohn-Sham Hamiltonian. 
The spectrum of this Hamiltonian can be approximated by solving eigenvalue problems arising from the Bloch transform.
To take advantage of the periodicity of the system,  
a plane-wave discretization is a natural choice. 
However the Coulomb potentials in the Hamiltonian considerably impede the convergence rate of Fourier methods because of the cusps \cite{kato1957eigenfunctions} of the eigenfunctions located at each nucleus.

Two approaches to deal with the Coulomb singularities can be distinguished. 
In the pseudopotential approach, the potential is regularized in a neighborhood of each nucleus, using a process which preserves the bottom of the spectrum of specific Hamiltonians (see \cite{cances2016existence,dupuy2018analysis} for more details on the generation of pseudopotentials).
In doing so, a modelling error is introduced which is rarely estimated.
Another way to handle the singularities is to use a different basis functions. 
Since the Coulomb potentials are point singularities, one can modify the plane-wave basis locally around the nuclei and use functions that are less sensitive to those singularities.
This is the main idea of the augmented plane-wave (APW) method \cite{slater1937wave} or of the linearised augmented plane-wave (LAPW) method \cite{koelling1975use}. 
As opposed to the pseudopotential approach, these methods dealing directly with Coulomb singularities are called \emph{full-potential} approach. 
The APW and LAPW methods can be reinterpreted as discontinuous Galerkin methods 
and variants of these methods have been studied in \cite{chen2015numerical,li2019discontinuous}.

The purpose of this paper is to provide an analysis of the variational projector-augmented wave (VPAW) method for the plane-wave discretization of periodic linear Hamiltonians with Coulomb potentials. 
The method has been introduced in \cite{blanc2017}. 
Its efficiency for the plane-wave discretization of one-dimensional toy model with Dirac potentials has been shown \cite{blanc2017vpaw1d}.
The VPAW method is based on the projector-augmented wave (PAW) method \cite{blochl94} which has become a state-of-the-art method in computational solid-state physics and implemented in several popular electronic structure simulation codes (Abinit \cite{torrent2008337}, VASP \cite{kresse1996efficient}).
The idea of the PAW method is to modify the original eigenvalue problem by an invertible transformation. 
This transformation acts locally in the neighborhood of each nucleus and maps atomic wave functions to smooth functions called pseudo wave functions. 
Moreover the PAW treatment allows for the use of pseudopotentials without any approximation. 
In practice however, the PAW method yields equations with infinite expansions that have to be truncated.
This induces an error that has been analyzed in \cite{blanc2017paw} for the same one-dimensional model mentioned previously.
In the VPAW method, the invertible transformation is a finite sum of operators, hence no truncation is needed and no approximation is made. 
Although Coulomb singularities are still present in the equations that are discretized, an acceleration of the plane-wave convergence can be proved.

The paper is organized as follows. 
The VPAW method applied to a periodic linear Hamiltonian is presented in Section~\ref{sec:vpaw_method}. 
Estimates on the eigenvalues of the plane-wave discretization of
the VPAW equations are given in Section~\ref{sec:main_theorem}. 
Numerical tests confirming the efficiency of the VPAW method on a simple model can be found in Section~\ref{sec:numerique}. 
Proofs of the acceleration of convergence are gathered in Section~\ref{sec:proofs}.

\section{The VPAW method}
\label{sec:vpaw_method}

\subsection{The eigenvalue problem}

For simplicity, we restrict ourselves to the linear model. A quick overview of the spectral theory of periodic
Hamiltonians can be found in \cite{gontier2015mathematical}. More thoroughful expositions of this theory are presented
in \cite{eastham1973spectral,kuchment2012floquet}. For extensions to nonlinear equations, the interested reader is
referred to \cite{catto2002some} for the Hartree model and to \cite{catto2001thermodynamic} for the Hartree-Fock model. 

The crystal is modelled as an infinite periodic motif of $N_{at}$ point charges at positions $\mathbf{R}_I$ in the unit cell 
$$
\Gamma = \left\{\alpha_1 \mathbf{a}_1 + \alpha_2 \mathbf{a}_2 + \alpha_3 \mathbf{a}_3, (\alpha_1, \alpha_2, \alpha_3) \in [-1/2, 1/2)^3 \right\}.
$$
 and repeated over a periodic lattice 
$$
\mathcal{R} = \Z \mathbf{a}_1 + \Z  \mathbf{a}_2 + \Z  \mathbf{a}_3.
$$
where $\mathbf{a}_1,\mathbf{a}_2,\mathbf{a}_3$ are linearly independent vectors of $\R^3$.  

The electronic properties of the crystal are determined by the spectrum of the periodic Hamiltonian $H_\mathrm{per}$ acting on $L^2(\R^3)$:
\begin{equation}
H_\mathrm{per} = -\frac{1}{2} \Delta + V_\mathrm{per} + W_\mathrm{per} ,
\end{equation}
where $V_\mathrm{per}$ is an $\mathcal{R}$-periodic potential defined by 
\begin{equation}
\label{eq:Vper}
\begin{cases}
-\Delta V_\mathrm{per} = 4\pi \left( \sumlim{\mathbf{T} \in \mathcal{R}}{} \sumlim{I=1}{N_{at}} Z_I \left(\delta_{\mathbf{R}_I} (\cdot + \mathbf{T}) - \frac{1}{|\Gamma|}\right) \right) \\
V_\mathrm{per} \text{ is } \mathcal{R}\text{-periodic}.
\end{cases}
\end{equation}
In this paper, $W_\mathrm{per}$ is a smooth $\mathcal{R}\text{-periodic}$ potential so that Equation \eqref{eq:Vper} has a solution. In practice, $W_\mathrm{per}$ is a nonlinear potential depending on the model chosen to describe the electronic self-interaction (typically a Kohn-Sham potential). 
\newline

The standard way to study the spectrum of $H_\mathrm{per}$ is through Bloch theory which will be outlined in the next few lines. Let $\mathcal{R}^*$ be the dual lattice 
$$
\mathcal{R}^* = \Z \mathbf{a}_1^* + \Z \mathbf{a}_2^* + \Z \mathbf{a}_3^*, 
$$
where  $(\mathbf{a}_1^*, \mathbf{a}_2^*,\mathbf{a}_3^*)$ satisfies $\mathbf{a}_i \cdot \mathbf{a}_j^* = 2\pi \delta_{ij}$. The reciprocal unit cell is defined by 
$$
\Gamma^* = \left\{\alpha_1 \mathbf{a}_1^* + \alpha_2 \mathbf{a}_2^* + \alpha_3 \mathbf{a}_3^*, (\alpha_1, \alpha_2, \alpha_3) \in [-1/2, 1/2)^3 \right\}.
$$
As $H_\mathrm{per}$ commutes with $\mathcal{R}$-translations, $H_\mathrm{per}$ admits a Bloch decomposition in operators $H_\mathbf{q}$ acting on 
$$
L^2_\mathrm{per}(\Gamma) = \{ f \in L^2_\mathrm{loc}(\R^3) \ | \ f \ \text{is } \mathcal{R}\text{-periodic}\},
$$ 
with domain 
$$
H^2_\mathrm{per}(\Gamma) = \{ f \in H^2_\mathrm{loc}(\R^3) \ | \ f \ \text{is } \mathcal{R}\text{-periodic}\}.
$$
The operator $H_\mathbf{q}$ is given by:
\[
H_\mathbf{q} = \frac{1}{2} |-i\nabla + \mathbf{q}|^2 + V_\mathrm{per} + W_\mathrm{per}, \quad \mathbf{q} \in \Gamma^*.
\]
For each $\mathbf{q} \in \Gamma^*$, the operator $H_\mathbf{q}$ is self-adjoint, bounded below and with compact resolvent. Thus it has a discrete spectrum of infinite eigenvalues $E_{1,\mathbf{q}} \leq E_{2,\mathbf{q}} \leq \dots \leq E_{n, \mathbf{q}} \to \infty$, counted with multiplicities, and the associated eigenfunctions $(\psi_{n, \mathbf{q}})_{n \in \N^*}$ form an orthonormal basis of $L^2_\mathrm{per}(\Gamma)$:
\begin{equation}
\label{eq:H_per}
H_\mathbf{q} \psi_{n,\mathbf{q}} = E_{n,\mathbf{q}} \psi_{n,\mathbf{q}}.
\end{equation}

By Bloch theorem \cite[Chapter XIII]{reed1978iv}, the spectrum of $H_\mathrm{per}$ is given by the union of the discrete spectra of an infinite number of eigenvalue problems parameterized by the vector $\mathbf{q}$ belonging to the reciprocal unit cell $\Gamma^*$:
\begin{equation}
\sigma (H_\mathrm{per} ; L^2(\R^3)) = \bigcup\limits_{\mathbf{q} \in \Gamma^*} \sigma ( H_\mathbf{q} ; L^2_\mathrm{per}(\Gamma)). 
\end{equation}

The VPAW method aims to ease the resolution of the eigenvalue problem \eqref{eq:H_per}. For clarity, we will only present the case $\mathbf{q}=0$ and denote $H_0$ by $H$ as this example contains all the main difficulties encountered in the numerical resolution of Equation \eqref{eq:H_per}.

\subsection{The VPAW method for solids}
\label{subsec:vpaw_method_for_solids}

Following the idea of the PAW method, an invertible transformation $(\mathrm{Id}+T)$ is applied to the eigenvalue problem \eqref{eq:H_per}, where $T$ is the sum of operators $T_I$ acting locally around each nucleus. For each operator $T_I$, two parameters $N_\mathrm{paw}$ and ${r_c}$ need to be fixed ($r_c$ and $N_\mathrm{paw}$ may depend on the atomic site $I$):
\begin{enumerate}
\item $N_\mathrm{paw}$ is the number of PAW functions used to build $T_I$,
\item ${r_c}$ is a cut-off radius which will set the acting domain of $T_I$, more precisely: 
\begin{itemize}
\item for all $f \in L_\mathrm{per}^2(\Gamma)$, $\mathrm{supp}(T_If) \subset \bigcup\limits_{\mathbf{T} \in \mathcal{R}} B(\mathbf{R}_I + \mathbf{T},{r_c})$, where $B(\mathbf{R},r)$ is the closed ball of $\R^3$ with center $\mathbf{R}$ and radius $r$,
\item if $\mathrm{supp}(f) \bigcap \bigcup\limits_{\mathbf{T} \in \mathcal{R}} B(\mathbf{R}_I + \mathbf{T},{r_c}) = \emptyset$, then $T_If = 0$.
\end{itemize} 
\end{enumerate}
The operator $T_I$ is given by:
\begin{equation}
\label{eq:T_0}
T_I = \sumlim{\mathbf{T} \in \mathcal{R}}{} \sum\limits_{k=1}^{N_\mathrm{paw}} (\phi^I_{k}(\mathbf{r} - \mathbf{R}_I) - \widetilde{\phi}^I_{k}(\mathbf{r} - \mathbf{R}_I)) \psh{\widetilde{p}^I_{k}(\cdot - \mathbf{R}_I)}{\bm \cdot},
\end{equation}
where $\psh{\bm \cdot}{\bm \cdot}$ is the $L^2$-scalar product on the unit cell $\Gamma$ and the functions $\phi^I_{k}$, $\widetilde{\phi}^I_k$ and $\widetilde{p}^I_k$ are functions in $L^2_\mathrm{per}(\Gamma)$. The PAW functions $(\phi_k^I)_{1 \leq k \leq N_\mathrm{paw}}$, $(\widetilde{\phi}_k^I)_{1 \leq k \leq N_\mathrm{paw}}$ and $(\widetilde{p}_k^I)_{1 \leq k \leq N_\mathrm{paw}}$ must satisfy the following essential properties:
\begin{enumerate}
\item $\mathrm{supp}\,(\phi_k^I - \widetilde{\phi}_k^I) \subset \bigcup\limits_{\mathbf{T} \in \mathcal{R}} B(\mathbf{T},{r_c}) $;
\item $\widetilde{\phi}_k^I$ restricted to $B(0,{r_c})$ is smooth;
\item $\widetilde{p}_k^I$ are supported in $\bigcup\limits_{\mathbf{T} \in \mathcal{R}} B(\mathbf{T},{r_c})$ and $\forall \, 1 \leq i,j \leq N_\mathrm{paw}, \  \psh{\widetilde{p}^I_i}{\widetilde{\phi}^I_j} = \delta_{ij}$ (\emph{i.e.} $(\widetilde{p}^I_j)_{1 \leq j \leq N_\mathrm{paw}}$ is dual to $(\widetilde{\phi}^I_j)_{1 \leq j \leq N_\mathrm{paw}}$).
\end{enumerate}
The operators $T_I$ act locally in  $\bigcup\limits_{\mathbf{T} \in \mathcal{R}} B(\mathbf{R}_I + \mathbf{T},{r_c})$. 

Several schemes exist in the literature to generate the PAW functions. In this paper, the PAW functions are close to the Vanderbilt scheme \cite{kresse99} where only the projector functions differ from ours. The Bl\"ochl scheme \cite{blochl94} is another popular way to generate PAW functions although the first seems to be preferred \cite{jollet2014generation}. See \cite{blanc2017vpaw1d,jollet2014generation} for more details on the generation of the PAW functions. 

\paragraph{Atomic wave function} Let $(\varphi^I_{k})_{1 \leq k \leq N_\mathrm{paw}} \in (L^2(\R^3))^{N_\mathrm{paw}}$ be eigenfunctions of an atomic \emph{non-periodic} Hamiltonian 
$$
H_I \varphi^I_k = \epsilon_k \varphi^I_k , \quad \epsilon^I_1 \leq \epsilon^I_2 \leq \epsilon^I_3 \leq \dots, \quad \int_{\mathbb{R}^3} \varphi^I_k \varphi^I_{k'} = \delta_{kk'},
$$
with $H_I$ defined by
\begin{equation}
\label{eq:atomic_hamiltonian}
H_I = -\frac{1}{2} \Delta - \frac{Z_I}{|\mathbf{r}|} + W_\mathrm{at}(|\mathbf{r}|),
\end{equation}
where $W_\mathrm{at}$ is a smooth bounded potential. The operator $H_I$ is
self-adjoint on $L^2(\R^3)$ with domain $H^2(\R^3)$. Again, in practice, $W_\mathrm{at}$ is a radial nonlinear potential belonging to the same family of models as $W_\mathrm{per}$ in Equation \eqref{eq:Vper}. Since the atomic Hamiltonian is rotationnaly invariant, $H_I$ is block-diagonal in the decomposition of $L^2(\R^3)$ associated with the eigenspaces of the operator $\mathbf{L}^2$ (the square of the angular momentum $\mathbf{L} = \mathbf{r} \times \mathbf{p} = \mathbf{r} \times (-i \nabla)$). The eigenfunctions $\varphi_k^I$ can be decomposed into a radial function and a spherical harmonics (see \cite[Chapter XIII.3.B]{reed1978iv}  for further details):
\begin{equation}
\label{eq:radial_phi}
\varphi_{k}^I(\mathbf{r}) = r^\ell R_{n \ell}(r) Y_{\ell m}(\hat{\mathbf{r}}),
\end{equation}
where $Y_{\ell m}$ is the real spherical harmonics and $R_{n\ell}$ is a continuous function such that $\lim\limits_{r \to 0} |R_{n\ell}(r)| < \infty$. For $\mathbf{r} \in \R^3$, we define $\hat{\mathbf{r}} := \frac{\mathbf{r}}{|\mathbf{r}|}$ and when there is no ambiguity we will denote by $r$ the euclidean norm of $\mathbf{r}$.
The decomposition \eqref{eq:radial_phi} also holds for some nonlinear models, see \cite{Solovej1991,cances2014mathematical}. The functions $R_{n\ell}$ satisfies the following radial Schr\"odinger equation 
\begin{equation}
\label{eq:radial_schrodinger}
\mathfrak{h}_\ell R_{n\ell}(r) =  -\frac{1}{2} R_{n\ell}^{\prime \prime} (r) - \frac{\ell+1}{r} R_{n\ell}'(r) - \frac{Z_I}{r} R_{n\ell}(r) + W_\mathrm{at}(r) R_{n\ell}(r) = \epsilon_{n\ell} R_{n\ell}(r).
\end{equation}
The eigenvalues of $\mathfrak{h}_\ell$, if they exist, are all simple. The discrete spectrum of $H_I$ is then the collection of all the eigenvalues of the operators $\mathfrak{h}_\ell$, $\ell \geq 0$. 

The PAW atomic wave functions $(\phi^I_{k})_{1 \leq k \leq N_\mathrm{paw}} \in (L^2_\mathrm{per}(\Gamma))^{N_\mathrm{paw}}$ are then defined by
\begin{itemize}
\item for $1 \leq k \leq N_\mathrm{paw}$ and $\mathbf{r} \in \Gamma$, $\phi^I_k(\mathbf{r}) = \varphi^I_k(\mathbf{r})$, 
\item $\phi^I_k$ is $\mathcal{R}$-periodic.
\end{itemize}
If $W_\mathrm{at} \not=0$,  there is a minimal angular momentum $\ell_\mathrm{adm}$ for which $\mathfrak{h}_{\ell}$ for all $\ell \geq \ell_\mathrm{adm}$ has no eigenvalue (see \cite[Theorem XIII.8]{reed1978iv} for the linear case and \cite{Solovej1991} for the reduced Hartree-Fock extension). 
As an immediate consequence, PAW functions can only be selected for a finite range of angular momentum $\ell \leq \ell_\mathrm{adm}$.

We denote by $(n_0,n_1,\dots,n_{\ell_\mathrm{max}})$ the number of PAW functions for each admissible angular momentum, \emph{i.e.} there are $n_0$ PAW functions for the angular momentum $\ell=0, m=0$, $n_1$ PAW functions for $\ell=1$, $|m| \leq 1$, \dots 
The total number of PAW functions for one atomic site is thus given by $N_\mathrm{paw}= \sumlim{\ell=0}{\ell_\mathrm{max}} (2\ell+1)n_\ell$. 

\paragraph{Pseudo wave function}

The pseudo wave functions $\widetilde{\phi}_k^I$ are the $\mathcal{R}$-periodic functions given in the unit cell $\Gamma$ by:
\begin{equation}
\forall \, \mathbf{r} \in \Gamma, \ \widetilde{\phi}_k^I(\br) = r^{\ell} \widetilde{R}_{n \ell}(r)\, Y_{\ell m}(\hat{\mathbf{r}}).
\end{equation}
where $k$ stands for the multiple index $(n,\ell,m)$.  
The radial functions $\widetilde{R}_{n \ell}$, $1 \leq n \leq n_\ell, 0 \leq \ell \leq \ell_\mathrm{max} $ are polynomial inside the augmentation region $B(0,{r_c})$:
\begin{equation}
\widetilde{R}_{n \ell}(r) = 
\begin{cases}
\sumlim{k=0}{d} c_{2k} r^{2k}  & \text{for } 0 \leq r \leq {r_c}\\
R_{n \ell}(r) & \text{for } r > {r_c}
\end{cases}
\end{equation}
and the coefficients are chosen to match $R_{n \ell}$ and its first $(d-1)$ derivatives of $R_{n \ell}$ at ${r_c}$. 

\paragraph{Projector functions}

The projector functions $(\widetilde{p}_k^I)_{1 \leq k \leq N_\mathrm{paw}}$ chosen here are the $\mathcal{R}$-periodic functions given in the unit cell $\Gamma$ by:
\begin{equation}
\forall \, \mathbf{r} \in \Gamma, \ \widetilde{p}^I_{n \ell m}(\mathbf{r}) = r^\ell p_{n \ell}(r) Y_{\ell m}(\hat{\mathbf{r}}).
\end{equation}
The functions $p_{n \ell}$ for $ 0 \leq \ell \leq \ell_\mathrm{max}, 1 \leq n \leq n_\ell$ are defined by
\begin{equation}
p_{n \ell}(r) = \sumlim{n'=1}{n_\ell} \left( B_\ell^{-1} \right)_{nn'} \chi(r) \widetilde{R}_{n' \ell}(r),
\end{equation}
with $\chi$ a smooth positive cut-off function supported in $(0,{r_c})$ and
\begin{equation}
B_{\ell} = \left(\itg{0}{{r_c}}{\chi (r) \widetilde{R}_{n \ell}(r) \widetilde{R}_{n' \ell}(r) r^{2+2\ell}}{r}\right)_{ 1 \leq n, n' \leq n_\ell}. 
\end{equation}
By definition, the projector functions $(\widetilde{p}^I_{k})_{1 \leq k \leq N_\mathrm{paw}}$ are supported in $\bigcup\limits_{\mathbf{T} \in \mathcal{R}} B(\mathbf{T},{r_c})$ and form a dual family to the pseudo wave functions $(\widetilde{\phi}^I_{k})_{1 \leq k \leq N_\mathrm{paw}}$: $\psh{\widetilde{p}^I_{k}}{\widetilde{\phi}^I_{k'}} = \delta_{kk'}$.
\newline

The VPAW equations to solve are then:
\begin{equation}
\label{eq:VPAW_eigenvalue_problem}
H^\mathrm{VPAW} \widetilde{\psi} = E S^\mathrm{VPAW} \widetilde{\psi},
\end{equation}
where 
\begin{equation}
\label{eq:H_VPAW}
H^\mathrm{VPAW} = (\mathrm{Id}+T)^* H (\mathrm{Id}+T), \quad S^\mathrm{VPAW} = (\mathrm{Id}+T)^*(\mathrm{Id}+T),
\end{equation}
and 
\begin{equation*}
T = \sum\limits_{I=1}^{N_{at}} T_I.
\end{equation*} 
Thus if $(\mathrm{Id}+T)$ is invertible, the eigenfunctions of $H$ can be recovered by the formula
\begin{equation}
\label{eq:I+Tpsi}
\psi = (\mathrm{Id}+T) \widetilde{\psi},
\end{equation}
and the eigenvalues are identical to the original eigenvalue problem \eqref{eq:H_per}.

By construction, the operator $(\mathrm{Id}+T_I)$ maps the pseudo wave functions $\widetilde{\phi}^I$ to the atomic eigenfunctions $\phi^I$:
\begin{equation*}
(\mathrm{Id}+T_I)\widetilde{\phi}^I_{k}(\cdot - \mathbf{R}_I) = \phi^I _{k}(\cdot - \mathbf{R}_I),
\end{equation*}
so if locally around each nucleus, the function $\psi$ ``behaves'' like the atomic wave functions $\phi^I_{k}$, we can hope that the cusp behavior of $\psi$ is captured by the operator $T$. 
The VPAW eigenfunction $\widetilde{\psi}$ would therefore be smoother than $\psi$ and the plane-wave expansion of $\widetilde{\psi}$ would converge faster than the expansion of $\psi$. 

\subsection{Well-posedness of the VPAW method}

To be well-posed the VPAW method requires
\begin{enumerate}
\item for each $0 \leq \ell \leq \ell_\mathrm{max}$, the family of pseudo wave functions $(\widetilde{R}_{n\ell})_{1 \leq n \leq n_\ell}$ to be linearly independent in $[0,r_c]$, so that the projector functions $(p_{n\ell})_{1 \leq n \leq n_\ell}$ are well defined;
\item $(\mathrm{Id}+T)$ to be invertible.
\end{enumerate} 

To fulfill the first condition, the following assertion is assumed. 

\begin{assumption}
\label{assump:cR_free_family}
For all $0 < r_c < r_\mathrm{min}$ and each $0 \leq \ell \leq \ell_\mathrm{max}$, $(\cR^{(k)}(r_c))_{0 \leq k \leq n_\ell-1}$ is a linearly independent family, where $\cR$ is the vector of the functions $(R_{1\ell}, \dots, R_{n_\ell \ell})$.
\end{assumption}

This condition ensures that the family of pseudo-wave functions $\left( \widetilde{R}_{n\ell} \right)_{1 \leq n \leq n_\ell}$ is linearly independent. This assumption holds in the particular case of the hydrogenoid atom (see Lemma \ref{lem:assumption_der_cR} in the appendix). 
\newline

It can be shown that the second condition is equivalent to the invertibility of the matrix  $\left(\psh{p_{j\ell}}{r^\ell R_{k\ell}}\right)_{1 \leq j,k \leq n_\ell}$ for each $0\leq \ell \leq \ell_\mathrm{max}$. Since the proof of this statement is very close to the proof of Proposition~2.3 in \cite{blanc2017vpaw1d}, we will not reproduce it here. For the rest of the paper, we make the following assumption.

\begin{assumption}
\label{assump:2}
For all $0 < r_c < r_\mathrm{min}$ and any $0 \leq \ell \leq \ell_\mathrm{max}$, the matrix  $\left(\psh{p_{n\ell}}{R_{n'\ell}}\right)_{1 \leq n,n' \leq n_\ell}$ is invertible. 
\end{assumption}

Finally, it appears in the proof of Theorem \ref{theo:energy} that we need the following assumption. It is shown in the appendix that this assumption is satisfied for the hydrogenoid eigenfunctions. 

\begin{assumption}
\label{assump:radial_PAW_function_cusp}
For every $0 \leq \ell \leq \ell_\mathrm{max}$ and $0 \leq n \leq n_\ell$, the radial functions $R_{n \ell}$ defined in \eqref{eq:radial_phi} satisfies $R_{n\ell}(0) \not= 0$.
\end{assumption}

There is a natural interpretation to this condition. The poor convergence of plane-wave expansions in the computation of the eigenvalues of \eqref{eq:H_per} is due to the cusps of the molecular wave functions. Hence we want to use atomic wave functions with cusps to reduce them.
By the Kato cusp condition (see Theorem \ref{theo:kato_cusp} below), this is equivalent to $R_{n\ell}(0) \not= 0$.

\subsection{Computational cost of the VPAW method}

A detailed analysis of the computational cost of the PAW method can be found in \cite{levitt2015parallel}: the cost scales like
$\cO(N_\mathrm{at} N_\mathrm{paw} M + N_\mathrm{at}N_\mathrm{paw}^2+ M \log M)$ where $N_\mathrm{at}$ is the number of nuclei, $N_\mathrm{paw}$ is the number of PAW functions per atomic site and $M$ the number of plane-waves. Usually, $N_\mathrm{paw}$ is
chosen relatively small, but $M$ may be large, so it is important to avoid a computational cost of order $M^2$. 

In practice, we are interested in the cost of the computation of $H^\mathrm{VPAW} \widetilde{\psi}$ and $S^\mathrm{VPAW} \widetilde{\psi}$ where $\widetilde{\psi}$ is expanded in $M$ plane-waves as the generalized eigenvalue problem is solved by a conjugate gradient algorithm. 
We will only focus on $H^\mathrm{VPAW} \widetilde{\psi}$ since the analysis $S^\mathrm{VPAW} \widetilde{\psi}$ is similar. Let us split $H^\mathrm{VPAW}$ into four terms:
$$
H^\mathrm{VPAW} \widetilde{\psi} = H \widetilde{\psi} + P D_H P^T \widetilde{\psi} + H \left(\Phi - \widetilde{\Phi}\right) P^T \widetilde{\psi} + P H \left(\Phi - \widetilde{\Phi}\right)^T \widetilde{\psi} ,
$$
where $P$ is the $M \times N_\mathrm{at}N_\mathrm{paw}$ matrix of the projector functions, $H (\Phi - \widetilde{\Phi})$ the $M \times N_\mathrm{at}N_\mathrm{paw}$ matrix of the Fourier representation of the $N_\mathrm{at}N_\mathrm{paw}$ functions $H(\phi^I_i - \widetilde{\phi}^I_i)$, and $D_H$ is the $N_\mathrm{at}N_\mathrm{paw} \times N_\mathrm{at}N_\mathrm{paw}$ matrix $\psh{\phi^I_i - \widetilde{\phi}^I_i}{H (\phi^J_j - \widetilde{\phi}^J_j)}$. Note that $D_H$ is a block diagonal matrix, because balls of radius $r_c$ centered at different atomic site do not overlap.

The computational cost can be estimated as follows (the cost at each step is given in brackets):
\begin{enumerate}
\item $H \widetilde{\psi}$ is assembled in two steps. First, $-\frac{1}{2} \Delta \widetilde{\psi}$ is computed in $\cO (M)$ since the operator $\frac{1}{2} \Delta$ is diagonal in Fourier representation. For the potential $V$, apply an inverse FFT to $\widetilde{\psi}$ to have the real space representation of $\widetilde{\psi}$, multiply pointwise by $V$ and apply a FFT to the whole result ($\cO(M \log M)$);
\item for $P D_H P^T \widetilde{\psi}$, compute the $N_\mathrm{at}N_\mathrm{paw}$ projections $P^T \widetilde{\psi}$ ($\cO(MN_\mathrm{at}N_\mathrm{paw})$), then successively apply the matrices $D_H$ ($\cO(N_\mathrm{at}N_\mathrm{paw}^2)$ since $D_H$ is block-diagonal) and $P$ ($\cO(MN_\mathrm{at}N_\mathrm{paw})$);
\item for $P H (\Phi - \widetilde{\Phi})^T \widetilde{\psi}$, similarly apply successively $H(\Phi - \widetilde{\Phi})^T$ to $\widetilde{\psi}$ ($\cO(MN_\mathrm{at}N_\mathrm{paw})$) and $P$ to $H (\Phi - \widetilde{\Phi})^T \widetilde{\psi}$ ($\cO(MN_\mathrm{at}N_\mathrm{paw})$);
\item for $H (\Phi - \widetilde{\Phi}) P^T \widetilde{\psi}$, we proceed as in step 3.
\end{enumerate}
Thus, the total numerical cost is of order $\cO(MN_\mathrm{at}N_\mathrm{paw} + N_\mathrm{at}N_\mathrm{paw}^2 + M \log M)$ which is the same as for the PAW method. 

\subsection{Singular expansion}

It appears that the theory of weighted Sobolev spaces and the singular expansion of eigenfunctions of Hamiltonians with Coulomb potentials provides a nice framework to study the Fourier decay of the VPAW pseudo wave functions $\widetilde{\psi}$. The singular expansion gives a generalization of the Kato cusp condition \cite{kato1957eigenfunctions} to any order.
This theory is closely linked to the $b$-calculus of pseudodifferential operators developed by Melrose \cite{melrose93}. It has been applied successfully to characterize precisely  the behaviour of the electronic wave function close the nucleus \cite{flad2008asymptotic,hunsicker2008analysis} and used in the analysis of the muffin-tin and LAPW methods \cite{chen2015numerical}. It is also a key element of the analysis of the convergence of \emph{hp}-finite elements approximation for similar models \cite{maday2019analyticity,maday2019hpdiscontinuous}. The interested reader may refer to \cite{kozlov1997elliptic, egorov2012pseudo} for a detailed exposition of this theory.\\ 

For simplicity and without loss of generality, we assume that $\Gamma$ is the cube $[-\frac{1}{2},\frac{1}{2}]^3$.
We denote by $\mathcal{S} \subset \R^3$ the set of the positions of the nuclei
\[
\mathcal{S} = \lbrace \mathbf{R}_I + \mathbf{T}, \ I=1,\dots,N_\mathrm{at}, \ \mathbf{T} \in \mathcal{R} \rbrace.
\]
Let $\chi$ be a $\mathcal{R}$-periodic continuous function such that $\chi(\mathbf{R}_I+\mathbf{r}) = r$ for small $r$, $\chi \in C^\infty_\mathrm{loc}(\mathbb{R}^3 \setminus \mathcal{S})$.

\begin{definition}
Let $k \in \N$ and $\gamma \in \R$.
We define the $k$-th weighted Sobolev space with index $\gamma$ by
\begin{equation}
\cK^{k,\gamma}(\Gamma) = \left\lbrace u \in L^2_\mathrm{per}(\Gamma) : \chi^{|\alpha|-\gamma} \partial^\alpha u \in L^2_\mathrm{per}(\Gamma) \ \forall \ |\alpha | \leq k \right\rbrace.
\end{equation}
\end{definition}

Consider a subspace of functions with the asymptotic expansions
\begin{equation}
\label{eq:asymptotic}
\forall I=1,\dots,N_\mathrm{at},  u(\mathbf{r}+\mathbf{R}_I) \sim \sumlim{j \in \N}{} c^I_j(\hat{\mathbf{r}}) r^j \ \text{as } r \to 0,
\end{equation}
where $c^I_j$ belongs to the finite dimensional subspace $M_j = \mathrm{span} \lbrace Y_{\ell m}, 0 \leq \ell \leq j, |m| \leq \ell \rbrace$. 

\begin{definition}
\label{def:weighted_sobolev_with_asymtotic}
Let $k \in \N$ and $\gamma \in \R$. We define the weighted Sobolev spaces with asymptotic type \eqref{eq:asymptotic} by
\begin{equation}
\label{eq:weighted_sobolev}
\begin{split}
\mathscr{K}^{k,\gamma}(\Gamma) = \Bigg\lbrace u \in \mathcal{K}^{k,\gamma}(\Gamma) \, \bigg| \, \eta_N \in \mathcal{K}^{k,\gamma+N+1}(\Gamma) \text{ where } \eta_N \text{ is the } \Gamma\text{-periodic function defined in } \Gamma \text{ by} \\
\left. \forall N \in \N, \ \forall \, \mathbf{r} \in \Gamma, \ \eta_N(\mathbf{r}) = u(\mathbf{r}) - \sumlim{I=1}{N_\mathrm{at}} \omega(|\mathbf{r}-\mathbf{R}_I|) \sumlim{j=0}{N} c_j^I(\widehat{\mathbf{r-R}_I}) |\mathbf{r-R}_I|^j \right\rbrace,
\end{split}
\end{equation}
where $\omega$ is a smooth positive cutoff function, \emph{i.e.} $\omega = 1$ near 0 and $\omega = 0$ outside some neighbourhood of $0$. 
\end{definition}

Definition \ref{def:weighted_sobolev_with_asymtotic} slightly differs from the definition of the weighted Sobolev space given in \cite{chen2015numerical} (Equation (2.6)). However, our definition is consistent with the results that can be found in \cite{hunsicker2008analysis} (see Theorem I.1) and the original paper \cite{flad2008asymptotic} (see Proposition 1) from which the definition appearing in \cite{chen2015numerical} is taken. 

The expansion \eqref{eq:asymptotic} can be viewed as a ``regularity expansion''. Let us suppose that the functions $c_j$ in the singular expansion are constant. Then all the even terms appearing in \eqref{eq:weighted_sobolev} are smooth since for any $k \in \N$, $r \mapsto r^{2k}$ is smooth. 
For the odd terms in the expansion, the function $r \mapsto r$ is continuous but not differentiable at the origin, the function $r \mapsto r^3$ is $C^2$ but not $C^3$ and so on. 
Since the decay of the Fourier coefficients depends on the regularity of the function, this expansion enables one to characterize precisely this decay. A precise estimation of this decay for all the terms appearing in \eqref{eq:weighted_sobolev} is given in Lemma \ref{lem:spherical_bessel} below.

\begin{definition}
A function $u$ is asymptotically well-behaved if $u \in \mathscr{K}^{\infty,\gamma}(\Gamma)$ for $\gamma < 3/2$. 
\end{definition}

\begin{note}
\label{rem:remainder_reg}
Note that if $u$ is asymptotically well-behaved then by the definition of the weighted Sobolev space with asymptotic type \eqref{eq:asymptotic}, the remainder
$\eta_N(\mathbf{r}) = u(\mathbf{r}) - \sumlim{I=1}{N_\mathrm{at}} \omega(|\mathbf{r}-\mathbf{R}_I|) \sumlim{j=0}{N} c^I_j(\widehat{\mathbf{r-R}_I}) |\mathbf{r-R}_I|^j $ is in the classical Sobolev space $H^{5/2 + N -\varepsilon}_\mathrm{per}(\Gamma)$. 
\end{note}

The following result, stated in \cite{hunsicker2008analysis,chen2015numerical}, gives the regularity of the eigenfunction of \eqref{eq:H_per} in terms of the previously defined weighted Sobolev space.

\begin{theo}[\cite{hunsicker2008analysis,chen2015numerical}]
\label{theo:psi_well_behaved}
Let $\psi$ be an eigenfunction of $H \psi = E\psi$ where $H$ is defined in \eqref{eq:H_per}. Then $\psi$ is asymptotically well-behaved. 
\end{theo}

Theorem \ref{theo:psi_well_behaved} enables to characterize precisely the singularity of the Hamiltonian wave function and generalizes the Kato cusp condition for eigenfunctions of 3D-Hamiltonians. 
Let $V^I$ be the smooth potential such that in a neighborhood of $\mathbf{R}_I$, 
\begin{equation}
\label{eq:external_V}
V_\mathrm{per}(\mathbf{r}) + W_\mathrm{per}(\mathbf{r}) = - \frac{Z}{|\mathbf{r}-\mathbf{R}_I|} + V^I(\mathbf{r}-\mathbf{R}_I), 
\end{equation}
and denote by $(v^I_k)_{k \geq \ell}$ the coefficients of the Taylor expansion of 
\begin{equation}
\label{eq:V_ellm}
V^I_{\ell m}(r) = \itg{S(0,1)}{}{V^I(\mathbf{r}) Y_{\ell m}(\hat{\mathbf{r}}) }{\hat{\mathbf{r}}}.
\end{equation}

\begin{theo}
\label{theo:kato_cusp}
Let $\ell \in \N$ and $|m| \leq \ell$. Let $\psi$ be an eigenfunction of $H \psi = E\psi$ where $H$ is defined in \eqref{eq:H_per} and $(\psi_{j\ell m})_{j,\ell \leq j, |m|\leq \ell}$ be the coefficients of the singular expansion of $\psi$, \emph{i.e.} for all $\varepsilon > 0$, 
$$
\psi(\mathbf{r}) - \sum\limits_{I=1}^{N_\mathrm{at}} \omega(|\mathbf{r-R}_I|) \sumlim{j=0}{N} \sumlim{|m| \leq \ell \leq j}{} \psi^I_{j\ell m} |\mathbf{r-R}_I|^j Y_{\ell m}(\widehat{|\mathbf{r-R}_I|}) \in \mathscr{K}^{\infty,\frac{5}{2}+N-\varepsilon}(\Gamma).
$$
Let $(v^I_k)_{k \geq \ell}$ be the coefficients of the Taylor expansion of the function $V^I_{\ell m}$ defined in \eqref{eq:V_ellm}. Then the sequence $(\psi^I_{j\ell m})_{j \geq \ell}$ satisfies 
\begin{equation}
\label{eq:psi_j_recurrence}
\forall \, j \geq \ell, \frac{(j+1)(j+2+2\ell)}{2} \psi^I_{j+1,\ell m} = - Z\psi^I_{j\ell m} + (v^I*\psi^I)_{j-1} - E\psi^I_{j-1,\ell m},
\end{equation}
where $v^I*\psi^I$ denotes the convolution 
$$
(v^I*\psi^I)_k = \sumlim{j=\ell}{k} v^I_{k-j} \psi^I_{k \ell m}.
$$
\end{theo}

The proof of this theorem follows from the definition of the weighted Sobolev space $\cK^{\infty,\frac{5}{2}+N-\varepsilon}$ and the Coulomb singularity of the potential. For $\ell=0$, the Kato cusp condition is recovered since $\psi^I_{000} = \psi(\mathbf{R}_I)$ and $\psi^I_{100} = \left. \frac{\partial }{\partial r}\right|_{r=0} \itg{S(\mathbf{R}_I,1)}{}{\psi(\mathbf{r})Y_{00}(\hat{\mathbf{r}})}{\hat{\mathbf{r}}}$.



\section{Main results}
\label{sec:main_theorem}

We focus on the analysis of the VPAW method restricted to a set of PAW functions associated to the lowest angular momentum $\ell = m = 0$ (\emph{i.e.} $\ell_\mathrm{max} = 0$). 

From Definition \ref{def:weighted_sobolev_with_asymtotic} and the asymptotic expansion of the molecular wave function $\psi$, it is possible to identify the cause of the slow decay of the Fourier coefficients of $\psi$, which is the cusps at each nucleus. 
We can show that the cusp of the pseudo wave function $\widetilde{\psi}$ is significantly reduced by the VPAW method.
More precisely (see Proposition~\ref{prop:cusp}), if $n_0$ PAW functions associated to the angular momentum $\ell=m=0$ are used, then there is a constant $C$ independent of $r_c$ such that for all $I=1,\dots,N_\mathrm{at}$, for any $0 < r_c \leq r_\mathrm{min}$ and for all $\varepsilon >0$:
\[
\left| \frac{\partial }{\partial r} \bigg|_{r=0} \int_{S(\mathbf{R}_I,1)} \widetilde{\psi}(\mathbf{r}) \, \mathrm{d}\hat{\mathbf{r}} \right| \leq C {r_c}^{\min(2n_0,5)-\varepsilon}.
\]
The blow-up of the $d$-th derivative introduced at each sphere $S(\mathbf{R}_I,r_c)$ is controlled similarly. It is possible to show (see Proposition~\ref{prop:d-th derivative}) that there exists a constant $C$ independent of $r_c$ such that $I=1,\dots,N_\mathrm{at}$, for any $0 < r_c \leq r_\mathrm{min}$ and for all $\varepsilon >0$:
\[
\left| \left[\int_{S(\mathbf{R}_I,1)} \widetilde{\psi}^{(d)} (\mathbf{r}) \, \mathrm{d} \hat{\mathbf{r}} \right]_{r_c} \right| \leq \frac{C}{{r_c}^{d-1}}.
\]
From both estimates, the following plane-wave convergence for the computation of the eigenvalues with the VPAW method can be proved.

\begin{theo}
\label{theo:energy}
Let $E_M$ be an eigenvalue of the variational approximation of \eqref{eq:H_VPAW} in a plane-wave basis with wavenumber $|\mathbf{K}| \leq M$, with $n_0$ PAW functions associated to the angular momentum $\ell = 0, m=0$ with smoothness $d \geq n$ and cut-off radius $r_c$. Let $E$ be the corresponding exact eigenvalue. Under Assumptions \ref{assump:cR_free_family}, \ref{assump:2} and \ref{assump:radial_PAW_function_cusp}, there exists a constant $C >0$ independent of ${r_c}$ and $M$ such that for $M$ sufficiently large, for all $\varepsilon>0$, and for all $0 < r_c < r_\mathrm{min}$, we have
\begin{equation}
\label{eq:VPAW_convergence}
0 < E_M - E \leq C \left( \frac{{r_c}^{2 \min(2n_0,5)-2\varepsilon}}{M^3} +  \frac{{r_c}^{\min(2n_0,5)-\varepsilon}}{M^{4-\varepsilon} }  
 + \frac{1}{{r_c}^{2d-2}} \frac{1}{M^{2d-1}} + o \left( \frac{1}{M^{5-\varepsilon}} \right) \right).
\end{equation}
\end{theo}

The VPAW method does not erase the cusps appearing in the molecular wave function $\psi$, hence in the asymptotic regime, the plane-wave convergence rate is the same as the brute force discretization of the original eigenvalue problem. 
The prefactor ${r_c}^{2 \min(2n_0,5)-2\varepsilon}$ can be significantly reduced by taking a small cut-off radius $r_c$. However in doing so, the second prefactor $\frac{1}{{r_c}^{2d-2}}$ can become dominant in the eigenvalue error. 
Balancing both error terms gives an optimal cut-off radius equal to $r_\mathrm{opt}=(M^\frac{2d-4}{4n_0+2d-2})^{-1}$. For $n_0=2$ and $d=5$ (which are typical for PAW simulations) and $r_c= r_\mathrm{opt}$, both error terms behave like $\frac{1}{M^{6}}$. 

Theorem \ref{theo:energy} holds for \emph{any} eigenvalue of $H$ given in \eqref{eq:H_per}, however numerical tests provided in Section \ref{sec:numerique} are restricted to the ground-state eigenvalue. They suggest that the VPAW method can be an efficient strategy to solve accurately the eigenvalue problem \eqref{eq:H_per} (see Figure \ref{fig:vpaw-custom}). 


\begin{note}
By incorporating $n_1$ functions for each angular momentum $\ell =1$ and $m \in \lbrace -1,0,1 \rbrace$, we can improve the convergence estimate \eqref{eq:VPAW_convergence} to
\begin{multline}
\label{eq:better_energy_estimate}
\forall 0 < r_c < r_\mathrm{min}, \ 0 < E_M - E \leq C \left( \frac{{r_c}^{2 \min(2n_0,5)-2\varepsilon}}{M^3} + \frac{{r_c}^{\min(2n_0,5)-\varepsilon}}{M^{4-\varepsilon} } \right. \\
\left. + \frac{{r_c}^{2\min(2n_1,5)-\varepsilon}}{M^5} + \frac{1}{{r_c}^{2d-2}} \frac{1}{M^{2d-1}} + o \left( \frac{1}{M^{7-\varepsilon}} \right) \right).
\end{multline}
The only difference between \eqref{eq:VPAW_convergence} and \eqref{eq:better_energy_estimate} is the prefactor of $\frac{1}{M^5}$. In our example (Section \ref{sec:numerique}), improvements for the computation of ground-state are marginal and visible for large plane-wave cutoffs (see Figure \ref{fig:vpaw-custom}). 
However, introducing PAW functions for $\ell=1$ might be beneficial for higher eigenvalues where in a pre-asymtotic regime, the prefactor of $\frac{1}{M^5}$ may be preponderant.
\end{note}

It is interesting to compare the VPAW method convergence with another full-potential approach like the augmented plane-wave (APW) method \cite{slater1937wave}. 
In the APW method, instead of modifying the Hamiltonian, a different basis set is used which is not sensitive to the cusps resulting from the Coulomb interaction with the nuclei. 
The new basis set is defined by partitioning the unit cell $\Gamma$  into two types of regions (the so-called \emph{muffin-tin division}):
\begin{enumerate}[label=\roman*)]
\item balls $B(\mathbf{R}_I,r_c)$, $I=1,\dots,M$;
\item the remaining interstitial region $\mathcal{D}$.
\end{enumerate}
The basis functions consist of augmentations of plane-waves:
\begin{equation}
\label{eq:APW_basis_function}
\omega_\mathbf{K}(\mathbf{r}) = \begin{cases}
e^{i \mathbf{K} \cdot \mathbf{r}} & \text{in } \mathcal{D} \\
\sumlim{\ell=0}{\ell_\mathrm{max}} \sumlim{m=-\ell}{\ell} \alpha_{\ell m}^\mathbf{K} \chi_\ell(r_I) Y_{\ell m}(\widehat{\mathbf{r}_I}) & \text{in each } B(\mathbf{R}_I,r_c)
\end{cases}
\end{equation}
where $\mathbf{r}_I = \mathbf{r}-\mathbf{R}_I$. The coefficients $(\alpha_{\ell m}^\mathbf{K})_{|m| \leq \ell \leq \ell_\mathrm{max}}$ are set to match the spherical harmonics expansion of $e^{i \mathbf{K} \cdot \mathbf{r}} $ at the boundaries of the balls $B(\mathbf{R}_I,r_c)$.
These basis functions are however not continuous at the boundary of the balls $B(\mathbf{R}_I,r_c)$, hence they do not belong to $H^1(\Gamma)$: the APW method is a \emph{nonconforming} method.
In \cite{chen2015numerical}, the convergence of the APW  method for a particular choice of $\chi_\ell$ is studied where each $\chi_\ell$ is a polynomial of degree less than $N$. 
The authors showed that the error on the eigenvalues of the problem \eqref{eq:H_per} by the APW method is bounded by 
\[
\forall s > \frac{3}{2}, \ | E^\mathrm{APW}_\eta - E | \leq \frac{C_s}{\eta^{s-\frac{3}{2}}},
\]
where $\eta = \min(M,\ell_\mathrm{max},N)$.

Although this bound holds for any $s > \frac{3}{2}$, the prefactor depends on $s$ and this dependency is not explicit in the paper. 
Moreover, in most situations, $\eta$ is equal to the maximal angular momentum $\ell_\mathrm{max}$. 
Hence, increasing this parameter is more and more costly since it introduces $(2\ell+1)(N+1)$  basis functions in the nonconforming method. 
On the other hand, the convergence of the VPAW method is already very good for $n_0 \leq 2$ PAW functions for $\ell=0$ (see Figure~\ref{fig:vpaw-custom}).

\section{Numerical results}
\label{sec:numerique}

In this section, we present some numerical results applied to the Hamiltonian $H$ in $[-\tfrac{L}{2},\tfrac{L}{2}]^3$
\begin{equation}
\label{eq:numerical_hamiltonian}
H = -\frac{1}{2} \Delta - \frac{Z}{\left|\mathbf{r}-\frac{\mathbf{R}}{2}\right|} - \frac{Z}{\left|\mathbf{r}+\frac{\mathbf{R}}{2}\right|},
\end{equation}
with periodic boundary conditions. The lowest eigenvalue is sought using iterative schemes, hence we are interested in the cost of the matrix-vector multiplication. 

The problem is solved using plane-waves. The kinetic operator is diagonal in the reciprocal space. The potential is discretized using a radial grid around the nuclei and a uniform grid in the rest of the domain. For the VPAW method, the following integrals are pre-computed:
\begin{enumerate}
\item $\psh{e_\mathbf{K}}{\widetilde{p}}$ : since $\widetilde{p}(\mathbf{r}) = p(r) Y_{\ell m}(\hat{\mathbf{r}})$  using \eqref{eq:Fourier-spherical-harm}, $\psh{e_\mathbf{K}}{\widetilde{p}}$ can be evaluated on a radial grid. 
\item $\psh{e_\mathbf{K}}{\phi-\widetilde{\phi}}$ : we proceed like for $\psh{e_\mathbf{K}}{\widetilde{p}}$ using a radial grid; 
\item $\psh{e_\mathbf{K}}{H(\phi-\widetilde{\phi})}$ : $H(\phi-\widetilde{\phi})$ is decomposed into a radial and a non-radial part. The radial part is evaluated on a radial grid and the non-radial part on a uniform grid. 
For non-linear approximations (Hartree-Fock and Kohn-Sham DFT), this term can be critical since it may be necessary to re-compute these integrals at each iteration. This is the main drawback of the VPAW method compared to the PAW method where this term does not exist.
\item $\psh{\phi-\widetilde{\phi}}{\phi-\widetilde{\phi}}$, $\psh{\phi-\widetilde{\phi}}{H(\phi-\widetilde{\phi})}$ : these integrals are computed using radial grids when possible or using 3D integration schemes. For nonlinear models, the last integral needs to be recomputed at each iteration, however, since there are $N_{\mathrm{paw}}^2$ of them, it is not too costly. 
\end{enumerate}

The numerical results using a Julia \cite{bezanson2017julia} homemade code are summarized in the following figures with $Z=3$, $R = 1$ and $L=5$. The atomic PAW function $\phi_k$ are the eigenfunctions of the hydrogenoid atom. 
For the pseudo atomic function $\widetilde{\phi}_k$, continuity of the function and of the first four derivatives are enforced (\emph{i.e.} $d=5$).
The lowest eigenvalue is computed using a conjugate-gradient algorithm stopped when the norm of the residual is less than $10^{-5}$.

For all the plots presented in this section, \emph{VPAW 1s} denotes the VPAW method with one PAW function for $\ell=0$ per atom, \emph{VPAW 2s} with two PAW functions for $\ell=0$ per atom and \emph{VPAW 2s1p} with two PAW functions for $\ell=0$ and one function for $\ell=1$, $|m|\leq 1$ per atom. 
The reference value for the lowest eigenvalue is given by the VPAW method for 200 plane waves per direction. Computation of the reference is out of reach by a direct plane-wave method. 
Figures \ref{fig:vpaw-custom}, \ref{fig:vpaw-cut-off-radius} and \ref{fig:vpaw-vs-paw} are log-log plots of the convergence of the lowest eigenvalue of \eqref{eq:numerical_hamiltonian} with respect to the number of plane-waves per direction for different choices of the PAW parameters.
Figure~\ref{fig:VPAW_rc_dependence_fixed_M} is a log-log plot of the energy difference
\newline

In Figure \ref{fig:vpaw-custom}, convergence rates of the energy for the VPAW method for a cutoff radius $r_c=\tfrac{R}{2}=0.5$ are presented. 
We clearly notice that the VPAW method converges faster than the direct method, gaining up to three orders of magnitude compared to a brute force plane-wave discretization. 
The convergence seems marginally faster when increasing the number of PAW functions per atom. This can be explained by the relatively large cut-off radius chosen for this example. 
We can distinctly see two regimes:
\begin{itemize}
    \item for $M \leq 20$, in the pre-asymptotic regime, the error on the eigenvalue is dominated by the term $\frac{1}{{r_c}^{2d-2}M^{2d-1}}$ as the estimated convergence rate suggests (numerically we observe a convergence rate $M^{-7.8}$ and theoretically $M^{-9}$ is expected). Since the same regularity of the atomic pseudo wave function $\widetilde{\phi}_k$ is used, it is not surprising to witness a similar behavior in that regime.
    \item for $M\geq 20$, the convergence rate is close to $M^{-4}$ which is the next error term given by Theorem~\ref{theo:energy}.
\end{itemize}

\begin{figure}[H]
\centering
\includegraphics[width=0.5\textwidth]{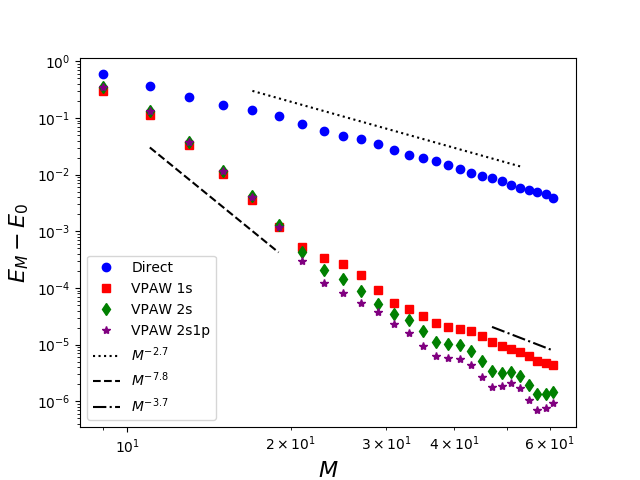}
\caption{Error on the lowest eigenvalue with the VPAW method for different choices of PAW functions.}
\label{fig:vpaw-custom}
\end{figure}

The size of the VPAW acting region can significantly impact the convergence rate in the pre-asymptotic regime (Figure \ref{fig:vpaw-cut-off-radius}). This plot suggests that there are different phases in the convergence of the VPAW eigenvalue:
\begin{enumerate}
\item for very low plane-wave cut-off ($M \leq 15$), the VPAW acting region is too small to be seen by the Fourier grid, hence no improvement is observed;
\item as the plane-wave cut-off grows, the VPAW eigenvalue converges very fast, since for this regime, the prefactors kill the $\frac{1}{M^{3}}$ and $\frac{1}{M^{4}}$ decay;
\item for a larger plane-wave cut-off $M \geq 50$, the convergence slows down since the prefactor for the  $\frac{1}{M^{4}}$ decay is not negligible anymore. Note that in this regime the error on the eigenvalue decreases as the VPAW cut-off radius $r_c$ is small, in agreement with Theorem~\ref{theo:energy}.
\end{enumerate}

\begin{figure}[H]
\centering
\includegraphics[width=0.5\textwidth]{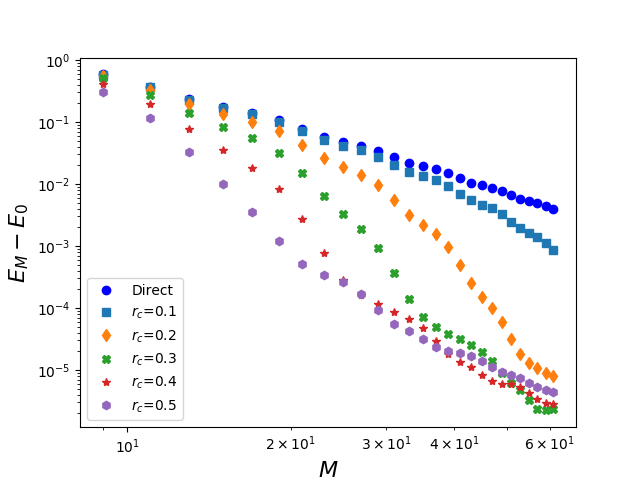}
\caption{Error on the lowest eigenvalue with the VPAW method for different choices of cut-off radius.}
\label{fig:vpaw-cut-off-radius}
\end{figure}

Figure~\ref{fig:VPAW_rc_dependence_fixed_M} illustrates the dependence on the VPAW cut-off radius $r_c$ of the error for the VPAW 1s method. The numerically observed behavior of the error is approximately the one predicted by Theorem~\ref{theo:energy} ($r_c^{-8}$ for small $M$ and $r_c^2$ for large $M$)

\begin{figure}[H]
    \centering
    \begin{subfigure}[t]{0.49\textwidth}
        \centering
        \includegraphics[width=\textwidth]{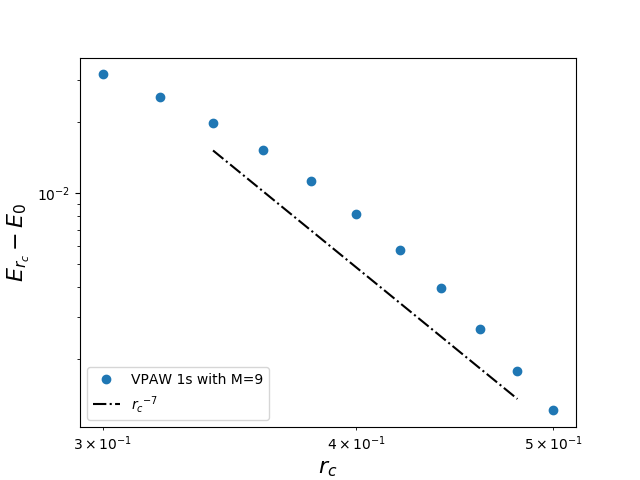}
        \caption{VPAW 1s with $M=9$}
    \end{subfigure}
    \begin{subfigure}[t]{0.49\textwidth}
        \centering
        \includegraphics[width=\textwidth]{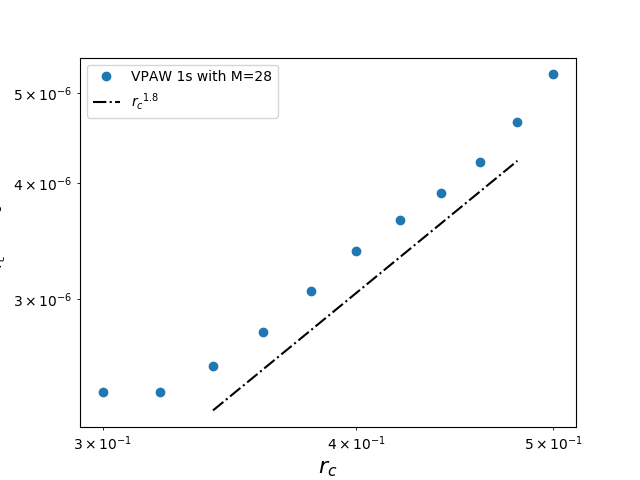}
        \caption{VPAW 1s with $M=28$}
    \end{subfigure}
    \caption{Error on the lowest eigenvalue with respect to the cut-off radius $r_c$ for a fixed plane-wave cut-off $M$}
    \label{fig:VPAW_rc_dependence_fixed_M}
\end{figure}

In Figure \ref{fig:vpaw-vs-paw}, a comparison between the original PAW method and the VPAW method is provided. 
Notice that the convergence of the eigenvalue for the PAW method is not monotone because the limit is below $E_0$.
If very accurate results are awaited on the lowest eigenvalue of \eqref{eq:numerical_hamiltonian}, the VPAW method seems the method of choice. 

\begin{figure}[H]
\centering
\includegraphics[width=0.5\textwidth]{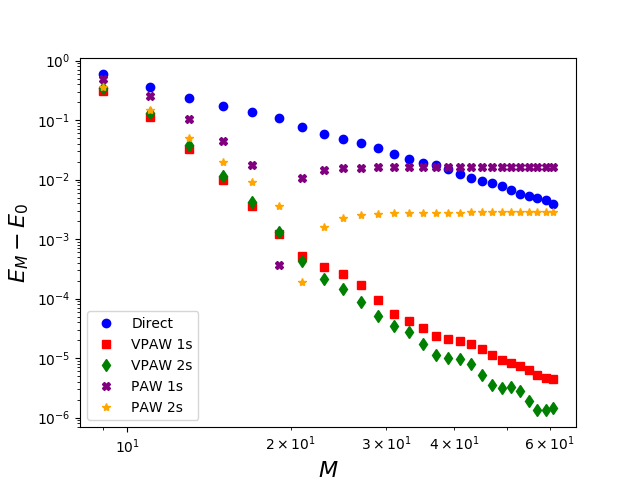}
\caption{Error on the lowest eigenvalue with the PAW and VPAW methods.}
\label{fig:vpaw-vs-paw}
\end{figure}

\section{Concluding remarks}

This paper provides an analysis of the VPAW method for the plane-wave discretization of eigenvalue problems of periodic linear Hamiltonians with Coulomb potentials. 
It theoretically and numerically shows the efficiency of the method to deal with Coulomb type singularities. 
Numerical tests to Kohn-Sham Hamiltonians are in preparation to assess the method on models closer to practice.
This work also opens the way of an analysis of the PAW method for 3D models. 
The PAW and VPAW methods are closely related. Since the VPAW method gives the same eigenvalues as the original Hamiltonian, using the connection between the PAW and VPAW methods should give an estimation of the error introduced by the PAW method. 
This approach has already proven to be successful in the analysis of this error for a one-dimensional model \cite{blanc2017paw}.

\section{Proof of Theorem~\ref{theo:energy}}
\label{sec:proofs}

The general idea of the proof is to isolate the main convergence difficulty which is the cusp of the eigenfunction $\psi$ located at each nucleus and see how the VPAW method reduces it. Although the VPAW method reduces the cusps of the VPAW eigenfunction $\widetilde{\psi}$, it introduces a derivative jump on the spheres $\partial B(\mathbf{R}_I,r_c)$ that blows up as the cut-off radius shrinks. 
As in \cite{blanc2017vpaw1d}, we split the pseudo wave function $\widetilde{\psi}$ into three parts using the singular expansion \eqref{eq:weighted_sobolev}.
Let $\omega$ be a smooth nonnegative cut-off function such that $\omega(r) = g(\frac{r}{r_c}) $ where $g$ satisfies
\begin{itemize}
\item $g$ is equal to 1 in $B(0,1/4)$,
\item $\mathrm{supp}(g) \subset B(0,1/2)$. 
\end{itemize}
Therefore, $\omega$ satisfies $\mathrm{supp}(\omega) \subset B(0,r_c/2)$, $\mathrm{supp}(1-\omega)^c \subset B(0,r_c/4)$ and $\|\omega^{(k)}\|_{L^\infty(0,r_c)} \leq \frac{C}{{r_c}^k}$.
Let $\eta \in \cK^{\infty,\frac{5}{2}+N-\varepsilon}(\Gamma)$ be the remainder of the singular expansion \eqref{eq:weighted_sobolev} applied to $\psi$:
\begin{equation*}
\eta(\mathbf{r}) = \psi(\mathbf{r}) - \sumlim{I=1}{N_\mathrm{at}} \omega(|\mathbf{r}-\mathbf{R}_I|) \sumlim{j=0}{N} c_j^I(\widehat{\mathbf{r-R}_I}) |\mathbf{r-R}_I|^j,
\end{equation*}
where by Theorem~\ref{theo:kato_cusp} $c^I_j = \sumlim{|m| \leq \ell \leq j}{} \psi^I_{j\ell m} Y_{\ell m}$.
By Equation~\eqref{eq:I+Tpsi}, we have:
\begin{align}
\widetilde{\psi}(\mathbf{r}) & = \psi(\mathbf{r}) - \sumlim{I=1}{N_\mathrm{at}} \sumlim{k=1}{N_\mathrm{paw}} \psh{\widetilde{p}_k^I}{\widetilde{\psi}}  (\phi^I_k(\mathbf{r-R}_I) - \widetilde{\phi}^I_k(\mathbf{r-R}_I)), \nonumber \\
& =\sumlim{I=1}{N_\mathrm{at}} \omega(|\mathbf{r}-\mathbf{R}_I|) \sumlim{j=0}{N} c_j^I(\widehat{\mathbf{r-R}_I}) |\mathbf{r-R}_I|^j + \eta(\mathbf{r}) - \sumlim{I=1}{N_\mathrm{at}} \sumlim{k=1}{N_\mathrm{paw}} \psh{\widetilde{p}_k^I}{\widetilde{\psi}}(\phi^I_k(\mathbf{r-R}_I) - \widetilde{\phi}^I_k(\mathbf{r-R}_I)),\nonumber \\
& =\sumlim{I=1}{N_\mathrm{at}} \omega(|\mathbf{r}-\mathbf{R}_I|) \left( \sumlim{j=0}{N} c_j^I(\widehat{\mathbf{r-R}_I}) |\mathbf{r-R}_I|^j -\sumlim{k=1}{N_\mathrm{paw}} \psh{\widetilde{p}_k^I}{\widetilde{\psi}}(\phi^I_k(\mathbf{r-R}_I) - \widetilde{\phi}^I_k(\mathbf{r-R}_I)) \right) \nonumber \\
& \qquad \qquad + \sumlim{I=1}{N_\mathrm{at}} (1- \omega(|\mathbf{r}-\mathbf{R}_I|) \sumlim{k=1}{N_\mathrm{paw}} \psh{\widetilde{p}_k^I}{\widetilde{\psi}}(\phi^I_k(\mathbf{r-R}_I) - \widetilde{\phi}^I_k(\mathbf{r-R}_I)) + \eta(\mathbf{r}). \label{eq:psi_decomposition}
\end{align}

The first part corresponds to the cusp of the pseudo wave function in a neighborhood of a nucleus. The second part is the $d$-th derivative jump caused by the lack of regularity at the sphere. 
The last part is the remainder appearing in the singular expansion of the original wave function $\psi$. In this section, we analyze the decay of the Fourier coefficients of the three parts separately.
\newline

We study the plane-wave convergence of the error on the eigenvalue for the VPAW method where $\ell_\mathrm{max}=0$. In this setting, for our purpose, a singular expansion for $N=1$
is sufficient. Since only PAW functions for the angular momentum $\ell =0$, $m=0$ are considered, the $00$ index in the PAW functions is dropped. 

\begin{prop}
\label{prop:cusp}
Let $(c^I_j)_{0 \leq j \leq 1}$ be the functions of the singular expansion \eqref{eq:weighted_sobolev} of $\psi$ . Let $n \geq 1$ be the number of PAW functions associated to the angular momentum $\ell = 0, m=0$. Then there exists a positive constant $C$ independent of $r_c$ and $K$ such that for all $I=1,\dots,N_\mathrm{at}$, for all $\varepsilon >0$ and for $K$ sufficiently large,  
\begin{equation}
\forall \, 0 < {r_c} < r_\mathrm{min}, \ \left| \itg{\Gamma}{}{\omega(r) \left( \sumlim{j=0}{1}c^I_j(\hat{\mathbf{r}})r^j - \sumlim{k=1}{n} \psh{\widetilde{p}_k^I}{\widetilde{\psi}}(\phi^I_k(\mathbf{r}) - \widetilde{\phi}^I_k(\mathbf{r})) \right) e^{-i \mathbf{K}\cdot \mathbf{r}}}{\mathbf{r}} \right| \leq \frac{C{r_c}^{\min(2n,5)-\varepsilon}}{K^{4}} .
\end{equation}
\end{prop}

This proposition states that the cusp of the VPAW eigenfunction $\widetilde{\psi}$ is reduced by a factor of order ${r_c}^{\min(2n,5)-\varepsilon}$. 
The proof can be found in Section~\ref{subsec:proof_prop_cusp_reduction}.

\begin{prop}
\label{prop:d-th derivative}
Let $n \geq 1$ be the number of PAW functions associated to the angular momentum $\ell = 0, m=0$. There exists a positive constant $C$ independent of $r_c$ such that for all $I=1,\dots,N_\mathrm{at}$ and for $K$ sufficiently large we have
\begin{equation}
\forall \, 0 < {r_c} < r_\mathrm{min}, \ \left| \itg{\Gamma}{}{(1-\omega(r))\sumlim{k=1}{n} \psh{\widetilde{p}_k^I}{\widetilde{\psi}}(\phi^I_k(\mathbf{r}) - \widetilde{\phi}^I_k(\mathbf{r})) e^{-i \mathbf{K}\cdot \mathbf{r}}}{\mathbf{r}} \right| \leq \frac{C}{{r_c}^{d-1} K^{d+2}}.
\end{equation}
\end{prop}

By reducing the cusp at a nucleus, the VPAW method introduces a derivative jump but for a higher order derivative. 
The proof can be found in Section~\ref{subsec:proof_prop_dth_derivative}.

\begin{prop}
\label{prop:remainder}
Let $\eta$ be the remainder of the expansion \eqref{eq:psi_decomposition} for $N=1$. Let $\eta_{M}$ be the truncation to the wavenumber $M$ of the plane-wave expansion of $\eta$. Then for all $\varepsilon >0$, we have
\begin{equation}
\| \eta_{M} - \eta \|_{H^1_\mathrm{per}} \leq \frac{1}{M^{5/2-\varepsilon}} \|\eta\|_{H^{7/2-\varepsilon}_\mathrm{per}}.
\end{equation}
\end{prop}

This is a direct consequence of the regularity of the remainder of the expansion~\eqref{eq:weighted_sobolev} given by Theorem~\ref{theo:psi_well_behaved} (see Remark~\ref{rem:remainder_reg}).

\subsection{Proof of Proposition~\ref{prop:cusp}}
\label{subsec:proof_prop_cusp_reduction}

Without loss of generality, we assume that there is a nucleus located at $\mathbf{R}_I=0$. Since the estimate in Proposition~\ref{prop:cusp} and \ref{prop:d-th derivative} do not depend on the atomic site, we will drop the upper index $I$. 

The following notation is introduced
\begin{align*}
p({r}) & := (p_1({r}), \dots, p_{n}({r}))^T \in \R^n, \\
\widetilde{p}(\mathbf{r}) & := (\widetilde{p}_1(\mathbf{r}), \dots, \widetilde{p}_{n}(\mathbf{r}))^T \in \R^n, \\
\psh{\widetilde{p}}{f} & := \left(\psh{\widetilde{p}_1}{f}, \dots, \psh{\widetilde{p}_{n}}{f}\right)^T \in \R^n, \forall \, f \in L^2_\mathrm{per}(\Gamma),\\
\Phi(\mathbf{r}) & := (\phi_1(\mathbf{r}), \dots, \phi_{n}(\mathbf{r}))^T \in \R^n, \\
\widetilde{\Phi}(\mathbf{r}) & := (\widetilde{\phi}_1(\mathbf{r}), \dots, \widetilde{\phi}_{n}(\mathbf{r}))^T \in \R^n, \\
\cR(r) & := (R_1(r), \dots, R_n(r))^T  \in \R^n, \\
\widetilde{\cR}(r) & := (\widetilde{R}_1(r), \dots, \widetilde{R}_n(r))^T  \in \R^n.
\end{align*}
For a function $f \in L^2([-\frac{1}{2},\frac{1}{2}]^3)$, we denote by $f_{\ell m}$ the averaged function
\begin{equation}
\label{eq:def_f_ellm}
f_{\ell m}(r) = \itg{S(0,1)}{}{f(\mathbf{r}) Y_{\ell m}(\hat{\mathbf{r}})}{\hat{\mathbf{r}}}.
\end{equation}
We recall the following identity that will be extensively used in the rest of the paper:
\begin{equation}
\label{eq:Fourier-spherical-harm}
e^{-i\mathbf{K}\cdot \mathbf{r}} = 4 \pi \sumlim{|m| \leq \ell}{} i^\ell Y_{\ell m}(-\hat{\mathbf{K}}) Y_{\ell m}(\hat{\mathbf{r}}) j_\ell(Kr),
\end{equation}
where $j_\ell$ is the spherical Bessel function of the first kind. 

To prove Proposition~\ref{prop:cusp}, we start with a lemma that caracterizes the main difficulties in the plane-wave convergence of the molecular wave function. 

\begin{lem}
\label{lem:spherical_bessel}
Let $\ell$ and $j$ be integers such that $\ell \leq j$. Let $K>0$. Then if $j+\ell$ is even, asymptotically as $K \to \infty$, we have for any positive integer $n \geq j+3$, 
\begin{equation}
\itg{0}{\frac{r_c}{2}}{\omega(r) r^{j+2} j_\ell(Kr)}{r} = \frac{\beta_{j,\ell}}{K^{j+3}}  + o \left( \frac{1}{K^n} \right) ,
\end{equation}
where 
$$
\beta_{j,\ell} = (-1)^{(j+\ell)/2} (j-\ell+1)! \prod_{k=0}^{\ell} (j-\ell+1+2k),
$$
and if $j+\ell$ is odd, for any positive integer $n \geq j+3$,
\begin{equation}
\itg{0}{\frac{r_c}{2}}{\omega(r) r^{j+2} j_\ell(Kr)}{r} = o \left( \frac{1}{K^n} \right).
\end{equation}
\end{lem}

\begin{proof}
We prove the lemma by induction on $\ell$. Let $a_{j+2, \ell}$ be defined by
\begin{equation}
a_{j+2, \ell} = \itg{0}{\frac{r_c}{2}}{\omega(r) r^{j+2} j_\ell(Kr)}{r}.
\end{equation}

\paragraph{Initialization}
For $\ell=0$, we have 
$$
j_0(x) = \frac{\sin (x)}{x},
$$
hence for any $j \in \N$,
\begin{equation}
\label{eq:a_j_imaginary_part}
a_{j+2,0} = \frac{1}{K} \mathrm{Im} \left(\itg{0}{\frac{r_c}{2}}{\omega(r) r^{j+1}e^{iKr}}{r}\right).
\end{equation}
By integration by parts, we have
\begin{align*}
\itg{0}{\frac{r_c}{2}}{\omega(r) r^{j+1}e^{iKr}}{r} & = \frac{1}{iK} \underbrace{\left[- \omega(r) r^{j+1} e^{iKr}\right]_0^{\frac{r_c}{2}}}_{=0} - \frac{1}{iK} \itg{0}{\frac{r_c}{2}}{(\omega(r) r^{j+1})' e^{iKr}}{r} \\
& = - \frac{1}{iK} \itg{0}{\frac{r_c}{2}}{\omega'(r)r^{j+1} e^{iKr}}{r} - \frac{j+1}{iK} \itg{0}{\frac{r_c}{2}}{\omega(r)r^{j} e^{iKr}}{r}.
\end{align*}
The function $r \mapsto r^{j+1} \omega'(r)$ belongs to $C^\infty_c(0,\frac{r_c}{2})$ hence we have for any $n > j+2$
\begin{equation}
\itg{0}{\frac{r_c}{2}}{\omega(r) r^{j+1}e^{iKr}}{r}   = -\frac{(j+1)}{iK} \itg{0}{\frac{r_c}{2}}{\omega(r) r^{j} e^{iKr}}{r} + o \left( \frac{1}{K^n} \right) 
\end{equation}
By integrating by parts $j$ times and noticing that the functions $r \mapsto r^k \omega'(r)$, $k \in \N$ are in $C^\infty_c(0,\frac{r_c}{2})$, we obtain
\begin{equation}
\itg{0}{\frac{r_c}{2}}{\omega(r) r^{j+1}e^{iKr}}{r} = (-1)^{j+1} \frac{(j+1)!}{(iK)^{j+2}} + o \left( \frac{1}{K^n} \right) .
\end{equation}
Hence using \eqref{eq:a_j_imaginary_part}, if $j$ is even, $a_{j+2,0} = o \left( \frac{1}{K^n} \right)$ for all positive integer $n$, otherwise $a_{j+2,0} = (-1)^{j/2} \frac{(j+1)!}{K^{j+3}}$.

\paragraph{Iteration}
Using the recurrence relation
\begin{equation}
j_{\ell+1}(x) = - j_\ell'(x) + \frac{\ell j_\ell(x)}{x},
\end{equation}
we have the following recurrence relation on $(a_{j,\ell})$:
\begin{align}
a_{j+2,\ell+1} & = \itg{0}{\frac{r_c}{2}}{\omega(r) r^{j+2} j_{\ell+1}(Kr)}{r} \\
& = \itg{0}{\frac{r_c}{2}}{\omega(r) r^{j+2}\left( - j_\ell'(Kr) + \frac{\ell j_\ell(Kr)}{Kr} \right)}{r} \\
& = \frac{\ell}{K} \itg{0}{\frac{r_c}{2}}{\omega(r) r^{j+1} j_{\ell}(Kr)}{r} - \itg{0}{\frac{r_c}{2}}{\omega(r) r^{j+2} j_{\ell}'(Kr)}{r}. \label{eq:bessel_rec}
\end{align} 
By integration by parts, using that $r \mapsto r^{j+2}\omega'(r) \in C^\infty_c(0,\frac{r_c}{2})$, we have
\begin{align*}
\itg{0}{\frac{r_c}{2}}{\omega(r) r^{j+2} j_{\ell}'(Kr)}{r} & = -\frac{j+2}{K}  \itg{0}{\frac{r_c}{2}}{\omega(r)r^{j+1}j_\ell(Kr)}{r} + o \left( \frac{1}{K^n} \right).
\end{align*}
Thus, we have by iteration, 
\begin{align*}
a_{j+2,\ell+1} & = \frac{j+\ell+2}{K} \itg{0}{\frac{r_c}{2}}{\omega(r)r^{j+1}j_\ell(Kr)}{r} + o \left( \frac{1}{K^n} \right)\\
& = \frac{j+\ell+2}{K} a_{j+1,\ell} + o \left( \frac{1}{K^n} \right)\\
& =  \frac{\prod_{k=0}^{\ell+1} (j-\ell+2k)}{K^{\ell+1}} a_{j-\ell+1,0} + o \left( \frac{1}{K^n} \right)  .
\end{align*}
Hence if $j+\ell+1$ is odd, $a_{j+2,\ell+1} = o\left( \frac{1}{K^n} \right)$ for all positive integer $n$, else if $j+\ell+1$ is even, for any $n \geq j+3$,
$$
a_{j+2,\ell+1} = (-1)^{(j-\ell+1)/2} (j-\ell)! \frac{\prod_{k=0}^{\ell+1} (j-\ell+2k)}{K^{j+3}} + o \left( \frac{1}{K^n} \right).
$$
\end{proof}

\begin{lem}
\label{lem:cusp}
Let $N \in \N^*$. Let $c_j$ be the functions of the singular expansion \eqref{eq:weighted_sobolev} at 0 of $\psi$. Let $\psi_{j\ell m}$ be the coefficients such that 
$$
c_j(\hat{\mathbf{r}}) = \sum_{\ell=0}^j \sum_{|m|\leq \ell} \psi_{j \ell m} Y_{\ell m}(\hat{\mathbf{r}}).
$$
Then, we have asymptotically as $K$ goes to $\infty$ and for any positive integer $n$,
\begin{equation}
\itg{[-\frac{1}{2},\frac{1}{2}]^3}{}{\omega(r) \sumlim{j=0}{N} c_j(\hat{\mathbf{r}}) r^j e^{-i \mathbf{K}\cdot \mathbf{r}}}{\mathbf{r}} = 4 \pi \sumlim{j=0}{N} \sum_{\substack{\ell=0 \\ j+\ell \ \textrm{odd}} }^j \sum_{|m|\leq \ell} i^\ell Y_{\ell m}(- \hat{\mathbf{K}}) \frac{\beta_{j,\ell} \psi_{j \ell m}}{K^{j+3}} + o \left( \frac{1}{K^n} \right).
\end{equation}

\end{lem}

\begin{proof}
We have since $\mathrm{supp}(\omega) \subset (0,\frac{r_c}{2})$
\begin{align}
\itg{[-\frac{1}{2},\frac{1}{2}]^3}{}{\omega(r) \sumlim{j=0}{N} c_j(\hat{\mathbf{r}}) r^j e^{-i \mathbf{K}\cdot \mathbf{r}}}{\mathbf{r}} & = \itg{0}{\frac{r_c}{2}}{r^2 \omega(r) \sumlim{j = 0}{N} r^j \sum_{\ell=0}^j \sum_{|m|\leq \ell} \psi_{j \ell m} \itg{S(0,1)}{}{Y_{\ell m}(\hat{\mathbf{r}})  e^{-i \mathbf{K}\cdot \mathbf{r}}}{\hat{\mathbf{r}}}}{r}.
\end{align}
Using the scattering expansion \eqref{eq:Fourier-spherical-harm} and applying Lemma \ref{lem:spherical_bessel}, we get
\begin{align}
\itg{[-\frac{1}{2},\frac{1}{2}]^3}{}{\omega(r) \sumlim{j=0}{N} c_j(\hat{\mathbf{r}}) r^j e^{-i \mathbf{K}\cdot \mathbf{r}}}{\mathbf{r}} & = 4 \pi \sumlim{j=0}{N}\sum_{\ell=0}^j \sum_{|m|\leq \ell}  i^\ell Y_{\ell m}^*(-\hat{\mathbf{K}}) \itg{0}{\frac{r_c}{2}}{\omega(r) r^{j+2} \psi_{j \ell m} j_\ell(Kr)}{r}  \\
& = 4 \pi \sumlim{j=0}{N} \sum_{\substack{\ell=0 \\ j+\ell \ \textrm{odd}} }^j \sum_{|m|\leq \ell} i^\ell Y_{\ell m}(- \hat{\mathbf{K}}) \frac{\beta_{j,\ell} \psi_{j \ell m}}{K^{j+3}} + o \left( \frac{1}{K^n} \right).
\end{align}
\end{proof}


According to Lemma \ref{lem:cusp}, the slowest decaying term is the term associated to $j=1$ and $\ell =0$. Proposition~\ref{prop:cusp} is then a direct consequence of Lemmas~\ref{lem:cusp} and \ref{lem:approx} stated below.

\begin{lem}
\label{lem:approx}
Let $\psi$ be an eigenfunction of $H\psi = E\psi$, with H defined in \eqref{eq:H_per}. Let $n$ be the number of PAW functions associated to the angular momentum $\ell=0,m=0$. 

Then there exist coefficients $(\alpha_k)_{1\leq k \leq n}$ and a positive constant $C$ independent of $r_c$ such that 
\begin{equation*}
\| \psi_{00} - \alpha^T \cR \|_{L^2(0,{r_c})} \leq C {r_c}^{1/2-\varepsilon + \min(2n,5)},
\end{equation*}
where $\psi_{00}$ denotes the averaged function $\psi$ according to \eqref{eq:def_f_ellm}. Moreover for these coefficients, we have 
$$
\psi_{00}(0) - \alpha^T \cR(0) = 0.
$$
\end{lem}

We first prove that Proposition~\ref{prop:cusp} follows from Lemmas~\ref{lem:cusp} and \ref{lem:approx}.

\begin{proof}[Proof of Proposition~\ref{prop:cusp}]
By definition of the PAW atomic functions $\phi_k = R_k Y_{00}$, these functions have a cusp at 0 and satisfy $R_k'(0) = -ZR_k(0)$. 
By Theorem~\ref{theo:psi_well_behaved}, we have $\psi_{100} = -Z\psi(0)$.
Using that the PAW pseudo wave functions are smooth in a neighbourdhood of 0 and Lemma~\ref{lem:cusp} for $N=1$, we obtain
\[
\left| \itg{\Gamma}{}{\omega(r) \left( \sumlim{j=0}{1}c^I_j(\hat{\mathbf{r}})r^j - \sumlim{k=1}{n} \psh{\widetilde{p}_k^I}{\widetilde{\psi}}(\phi^I_k(\mathbf{r}) - \widetilde{\phi}^I_k(\mathbf{r})) \right) e^{-i \mathbf{K}\cdot \mathbf{r}}}{\mathbf{r}} \right| \leq \frac{C}{K^4} |\psi(0) - \psh{\widetilde{p}}{\widetilde{\psi}}^T \mathcal{R}(0)|,
\]
where $C>0$ is a constant independent of $r_c$ and $K$.
We will now show that $|\psi(0) - \psh{\widetilde{p}}{\widetilde{\psi}}^T \mathcal{R}(0)| \leq C {r_c}^{\min(2n,5)-\varepsilon}$, for some constant $C>0$ independent of $r_c$. 
By duality of $\widetilde{\phi}_j$ and $\widetilde{p}_k$,  for any $\alpha \in \R^n$ we have
\begin{equation}
\psi(0) - \psh{\widetilde{p}}{\widetilde{\psi}}^T \mathcal{R}(0) = \psi_{00}(0) - \alpha^T \cR(0) - \psh{\widetilde{p}}{\widetilde{\psi} - \alpha^T \widetilde{\Phi} }^T \cR(0) \label{eq:proof_intermediaire}.
\end{equation}
First we rewrite $\psh{\widetilde{p}}{\widetilde{\psi} - \alpha^T \widetilde{\Phi} }$ in a more convenient way. In a neighbordhood of 0, we have 
$$
\psi - \psh{\widetilde{p}}{\widetilde{\psi}}^T \Phi = \widetilde{\psi} - \psh{\widetilde{p}}{\widetilde{\psi}}^T \widetilde{\Phi}.
$$
By multiplying by $\widetilde{p}_k$, $k=1,..,n$ and integrating over the ball $B(0,{r_c})$, we obtain 
$$
\psh{\widetilde{p}_k}{\psi}  - \psh{\widetilde{p}}{\widetilde{\psi}}^T \psh{\widetilde{p}_k}{\Phi} = 0,
$$
so
$$
\psh{\widetilde{p}}{\widetilde{\psi}} = A^{-1}\psh{\widetilde{p}}{\psi},
$$
where $A = (\psh{\widetilde{p}_j}{\phi_k})_{1 \leq j,k \leq n} = (\psh{p_j}{R_k}_{[0,r_c]})_{1 \leq j,k \leq n}$ which is invertible by Assumption \ref{assump:cR_free_family}. By definition of the PAW functions, the 3D-integrals can be reduced to integrals on an interval 
$$
\psh{\widetilde{p}}{\widetilde{\psi}}  = A^{-1} \psh{p}{\psi_{00}}_{[0,{r_c}]} ,
$$
where 
\[
\psh{f}{g}_{[0,r_c]} = \itg{0}{r_c}{f(r)g(r)r^2}{r}.
\]
By duality of the PAW functions, $\psh{\widetilde{p}_j}{\widetilde{\phi}_k} = \delta_{jk}$, we have $\psh{\widetilde{p}}{\alpha^T \widetilde{\Phi}} = A^{-1} \psh{p}{\alpha^T \cR}_{[0,r_c]}$. Hence,
$$
\psh{\widetilde{p}}{\widetilde{\psi}-\alpha^T \widetilde{\Phi}} =  A^{-1} \psh{p}{\psi_{00}-\alpha^T \cR}_{[0,{r_c}]} .
$$
By Lemma \ref{lem:tilde_p_stuff}, there exists a constant $C$ independent of ${r_c}$ such that for any ${r_c} >0$,
\begin{equation}
\left| \psh{p}{\psi_{00} - \alpha^T \cR}_{[0,r_c]}^T A^{-T} \cR'(0) \right| \leq \frac{C}{{r_c}^{3/2}} \| \psi_{00} - \alpha^T \cR \|_{L^2(B(0,{r_c}))} \leq  \frac{C}{{r_c}^{1/2}} \| \psi_{00} - \alpha^T \cR \|_{L^2(0,{r_c})}.
\end{equation}
From Lemma \ref{lem:approx}, we know that there exists $\alpha \in \R^n$ such that
\begin{equation*}
\| \psi_{00} - \alpha^T \cR \|_{L^2(0,{r_c})} \leq C {r_c}^{\frac{1}{2}-\varepsilon+\min(2n,5)} \quad \text{ and } \quad \alpha^T \cR(0) = \psi(0).
\end{equation*}
Inserting this equation into \eqref{eq:proof_intermediaire} finishes the proof. 

\end{proof}

To show Proposition~\ref{prop:cusp}, it suffices to prove Lemma~\ref{lem:approx}. In order to show this lemma, we start with a few intermediary results and introduce some notation. Let $\uppsi_k \in \R$ and $\zeta_k \in \R^n$ be, respectively, the coefficients of the singular expansion of $\psi_{00}$ and $\cR$:
\begin{align}
\psi_{00}(r) & = \sumlim{j=0}{N} \uppsi_{j} r^j + \eta_{N+1}(r), \quad \eta_{N+1} \in \cK^{\infty, \frac{5}{2}+N -\varepsilon}(\Gamma) \\
\cR(r) &= \sumlim{j=0}{N} \zeta_j r^j + \xi_{N+1}(r), \quad \xi_{N+1} \in \cK^{\infty, \frac{5}{2}+N -\varepsilon}(\Gamma). \label{eq:cR_zeta_expansion}
\end{align}
The potential $V^I$ defined in Equation~\eqref{eq:external_V} is smooth, hence we can find $(v_{2k})_{0 \leq k \leq N} \in \R^N$ such that
\begin{align}
V_{00}^I(r) = \sumlim{k=0}{N} v_{2k} r^{2k} + \cO(r^{2N+2}).
\end{align}
The atomic potential $W_\mathrm{at}$ in Equation~\ref{eq:atomic_hamiltonian} is also smooth so there is $(w_{2k})_{0 \leq k \leq N} \in \R^N$ such that
\begin{equation}
\label{eq:coefficient_w_atomic_potential}
    (W_\mathrm{at})_{00}(r) = \sumlim{k=0}{N} w_{2k} r^{2k} + \cO(r^{2N+2}).
\end{equation}
Finally, we denote $\cE$ the diagonal matrix with entries $(\epsilon_1,\dots,\epsilon_n)$.

\begin{lem}
\label{lem:zeta_k}
Let $\cR = (R_1, \dots, R_n)^T$ where $R_k$ is defined in \eqref{eq:radial_phi}. 

There exists $(\mu_j^{(k)})_{0 \leq j \leq k \leq n-1}$ and $(\nu_j^{(k)})_{0 \leq j \leq k \leq n-1}$ such that
\begin{align}
\zeta_{2k} &  = \sumlim{j=0}{k} \mu_j^{(k)} \cE^j \zeta_0 \\
\zeta_{2k+1} &  = \sumlim{j=0}{k} \nu_j^{(k)} \cE^j \zeta_0 ,
\end{align}
with $\mu_k^{(k)} \not=0$ for any $0 \leq k \leq n-1$.

Moreover, the vectors $(\zeta_k)_{1 \leq k \leq 5}$ satisfies
\begin{align*}
\zeta_1 & = - Z\zeta_0 \\
\zeta_2 & = - \frac{1}{3} \cE \zeta_0 + \frac{1}{3} (Z^2 + w_0) \zeta_0 \\
\zeta_3 & = \frac{2}{9} Z \cE \zeta_0 - \left( \frac{Z^3}{18} + \frac{2}{9} Z w_0 \right) \zeta_0 \\
\zeta_4 & = \frac{1}{30} \cE^2 \zeta_0 - \left( \frac{Z^2}{18} + \frac{w_0}{15} \right) \cE \zeta_0 + \left( \frac{Z^4}{180} + \frac{Z^2w_0}{18} + \frac{w_0^2}{30} + \frac{w_2}{10} \right) \zeta_0 \\
\zeta_5 & = - \frac{23}{1350} Z \cE^2 \zeta_0 + \left( \frac{Z^3}{135} + \frac{23}{675} w_0 Z \right) \cE \zeta_0  - \left( \frac{Z^5}{2700} + \frac{Z^3 w_0}{135} + \frac{23}{1350} Z w_0^2 + \frac{11}{150} Z w_2\right) \zeta_0.
\end{align*}
\end{lem}

We give an explicit formula of $\zeta_k$ for $k \leq 5$ because we will show that we cannot systematically improve the bound in Lemma~\ref{lem:approx} for $n \geq 2$.

\begin{proof}
This lemma is proved by iteration using Equation~\eqref{eq:psi_j_recurrence} applied to each $R_k$, $k=1,\dots,n$
\begin{equation*}
\forall N \geq 1, \ \frac{(N+1)(N+2)}{2} \zeta_{N+1} = - Z \zeta_N + (w*\zeta)_{N-1} - \cE \zeta_{N-1}. 
\end{equation*}
Here $(w_{2k})_{k\geq 0}$ and $(\zeta_{2k})_{k \geq 0}$ are defined respectively by \eqref{eq:cR_zeta_expansion} and \eqref{eq:coefficient_w_atomic_potential}. Since $W_\mathrm{at}$ is smooth, we have 
\begin{equation*}
(w*\zeta)_{2k} = \sumlim{j=0}{k} w_{2k-2j} \zeta_{2j}, \quad (w*\zeta)_{2k+1} = \sumlim{j=0}{k} w_{2k-2j} \zeta_{2j+1}.
\end{equation*}
By iteration, one notices that $\mu_k^{(k)} \not= 0$ for all $k \in \N$. 
\end{proof}

\begin{proof}[Proof of Lemma \ref{lem:approx}]
To minimize $\|\psi_{00} - \alpha^T \cR\|_{L^2(0,r_c)}$ with respect to $\alpha$ for $r_c$ small, we need to determine how many successive terms in the singular expansion of $\psi_{00}$ can be canceled with $n$ functions $R_k, k=1,\dots,n$, \emph{i.e.} we need to determine for which $N_\mathrm{max} \leq 2n$ we have
\begin{equation}
\label{eq:alpha_zeta=psi}
\forall \, 0 \leq k \leq N_\mathrm{max}, \ \alpha^T \zeta_k = \uppsi_k.
\end{equation}
Expressing $\zeta_k$ (resp. $\uppsi_k$) as a linear combination of $(\zeta_0,\dots,\cE^{\lceil \frac{N_\mathrm{max}}{2} \rceil-1} \zeta_0)$ (resp. $(\uppsi_0,\dots,E^{\lceil \frac{N_\mathrm{max}}{2} \rceil-1}\uppsi_0)$) using Lemma \ref{lem:zeta_k}, the linear system \eqref{eq:alpha_zeta=psi} can be reformulated as 
\begin{equation}
\label{eq:linear_system}
\begin{pmatrix}
M_1 \\
M_2 
\end{pmatrix} \begin{pmatrix}
\zeta_0^T \\
\zeta_0^T\cE \\
\vdots \\
\zeta_0^T \cE^{\lceil N_\mathrm{max}/2 \rceil-1} 
\end{pmatrix} \alpha = \begin{pmatrix}
N_1 \\
N_2 
\end{pmatrix} \begin{pmatrix}
\uppsi_0 \\
E \uppsi_0 \\
\vdots \\
E^{\lceil N_\mathrm{max}/2 \rceil-1} \uppsi_0 
\end{pmatrix} ,
\end{equation}
where $M_1 = (\mu_j^{(k)})_{0 \leq j,k \leq \lceil N_\mathrm{max}/2 \rceil-1},$ and $M_2 = (\nu_j^{(k)})_{0 \leq j,k \leq \lceil N_\mathrm{max}/2 \rceil-1},$ with $\mu_j^{(k)}$ and $\nu_j^{(k)}$ given by Lemma \ref{lem:zeta_k}. $N_1$ and $N_2$ the same matrices as $M_1$ and $M_2$ but where the coefficients $\mu_j^{(k)}$ and $\nu_j^{(k)}$ are generated using $(v_{2k})$ instead of $(w_{2k})$. 

We will show that if $n \leq 2$, then $N_\mathrm{max}=2n$, otherwise, $N_\mathrm{max}=5$. Equation \eqref{eq:linear_system} is equivalent to 
\begin{equation*}
\begin{cases}
M_1 (\zeta_0^T \cE^k)_{0 \leq k \leq \lceil N_\mathrm{max}/2 \rceil-1} \alpha = N_1 (E^k \uppsi_0)_{0 \leq k \leq \lceil N_\mathrm{max}/2 \rceil-1} \\
M_2 (\zeta_0^T \cE^k)_{0 \leq k \leq \lceil N_\mathrm{max}/2 \rceil-1} \alpha = N_2 (E^k \uppsi_0)_{0 \leq k \leq \lceil N_\mathrm{max}/2 \rceil-1}, 
\end{cases}
\end{equation*}
hence Equation \eqref{eq:linear_system} has a solution if and only if 
\[
M_2 M_1^{-1} N_1 (E^k \uppsi_0)_{0 \leq k \leq \lceil N_\mathrm{max}/2 \rceil-1} = N_2(E^k \uppsi_0)_{0 \leq k \leq \lceil N_\mathrm{max}/2 \rceil-1}. 
\]
Since this holds for any value $E$, a necessary and sufficient condition to solve \eqref{eq:linear_system} is $M_2 M_1^{-1} N_1 = N_2$.
\newline

For $n=1$, $M_1 = N_1 = \begin{pmatrix} 1 \end{pmatrix}$ and $M_2 = N_2 = \begin{pmatrix} -Z \end{pmatrix}$ hence \eqref{eq:linear_system} is solvable when $N_\mathrm{max}=2$. 
The remainder $\psi_{00} - \alpha \cR$ belongs to $\cK^{\infty,\frac{7}{2}-\varepsilon}(\Gamma)$, hence by Lemma \ref{lem:cK_a_decay}, there exists a constant $C$ independent of $r_c$ such that
\[
\| \psi_{00} - \alpha \cR \|_{L^2(0,r_c)} \leq C {r_c}^{5/2-\varepsilon}.
\]

For $n=2$, we can check using Lemma \ref{lem:zeta_k} that
\[
M_2 M_1^{-1} = \begin{pmatrix}
-Z & 0 \\
\frac{Z^3}{6} & -\frac{2Z}{3}
\end{pmatrix}.
\]
Thus $M_2 M_1^{-1} N_1 = N_2$ and \eqref{eq:linear_system} has a solution  $(\alpha_1,\alpha_2)$ such that $\psi_{00}(r) - \alpha^T \cR(r)$ belongs to $\cK^{\infty,\frac{11}{2}-\varepsilon}(\Gamma)$, hence by Lemma \ref{lem:cK_a_decay}, 
$\| \psi_{00}(r) - \alpha^T \cR(r) \|_{L^2(0,r_c)} \leq C {r_c}^{9/2-\varepsilon}$.

For $n \geq 3$, the dependence on the atomic potential $W_\mathrm{at}$ does not vanish in $M_2 M^{-1}_1$. For example, the $(3,1)$ coordinate of the matrix $M_2M^{-1}_1$ has a term equal to $\frac{391}{6075} Z^3 w_0$ which is unlikely to be compensated in general. For $n \geq 3$, we thus have:
$$
\psi_{00}(r) - \alpha^T \cR(r) = \cO(r^5),
$$
hence $\| \psi_{00}(r) - \alpha^T \cR(r) \|_{L^2(0,r_c)} \leq C {r_c}^{11/2}$ for a constant $C$ independent of $r_c$.
\end{proof}

\subsection{Proof of Proposition~\ref{prop:d-th derivative}}
\label{subsec:proof_prop_dth_derivative}

The derivative jump of the $d$-th derivative jump at $r=r_c$ of $\psh{\widetilde{p}}{\widetilde{\psi}}^T \widetilde{\cR}^{(d)}$ needs to be estimated. 

\begin{lem}
\label{lem:d-th der jump}
There exists a positive constant $C$ independent of $r_c$ such that
\begin{equation*}
\forall \, 0 < {r_c} < r_\mathrm{min}, \ \left| \psh{\widetilde{p}}{\widetilde{\psi}}^T [\widetilde{\cR}^{(d)}]_{{r_c}} \right| \leq  \frac{C}{{r_c}^{d-1}} ,
\end{equation*}
where $[\widetilde{\cR}^{(d)}]_{{r_c}} = \lim\limits_{h \to 0, h>0} \widetilde{\cR}^{(d)}(r_c+h)-\widetilde{\cR}^{(d)}(r_c-h)$.
For any $k \in \N$, we have
\begin{equation*}
\left\| \psh{\widetilde{p}}{\widetilde{\psi}}^T \left(\cR^{(k)} - \widetilde{\cR}^{(k)} \right) \right\|_{L^\infty(0,r_c)} \leq \frac{C}{{r_c}^k}.
\end{equation*}
\end{lem}

The proof of this lemma is given in the appendix. 

\begin{proof}[Proof of Proposition \ref{prop:d-th derivative}]
Using the relation \eqref{eq:Fourier-spherical-harm}, we have
\begin{equation}
\itg{[-\frac{1}{2},\frac{1}{2}]^3}{}{(1-\omega(r))\psh{\widetilde{p}}{\widetilde{\psi}}^T (\Phi - \widetilde{\Phi})  e^{-i \mathbf{K}\cdot \mathbf{r}}}{\mathbf{r}} = 4 \pi \itg{0}{{r_c}}{(1-\omega(r)) \psh{\widetilde{p}}{\widetilde{\psi}}^T (\cR(r) - \widetilde{\cR}(r)) j_0(Kr)r^2}{r}.
\end{equation}
Since $\omega$ is equal to 1 in a neighbourhood of 0, we can restrict the integral in the equation above to the interval $({r_c} - \eta, {r_c})$ for some $\eta>0$. Recall that 
$$
j_0(x) = \frac{\sin(x)}{x},
$$
thus
\begin{equation}
\label{eq:dth-der1}
\itg{[-\frac{1}{2},\frac{1}{2}]^3}{}{(1-\omega(r))\psh{\widetilde{p}}{\widetilde{\psi}}^T (\Phi - \widetilde{\Phi})  e^{-i \mathbf{K}\cdot \mathbf{r}}}{\mathbf{r}} = \frac{4\pi}{K} \psh{\widetilde{p}}{\widetilde{\psi}}^T \itg{{r_c}-\eta}{{r_c}}{(1-\omega(r))  (\cR(r) - \widetilde{\cR}(r)) \sin(Kr)r}{r}.
\end{equation}
We denote by $f$ the function $r \mapsto r (1-\omega(r))  (\cR(r) - \widetilde{\cR}(r))$ and use 
$$
\itg{{r_c}-\eta}{{r_c}}{f(r) \sin(Kr)}{r} = \mathrm{Im} \left( \itg{{r_c}-\eta}{{r_c}}{f(r) e^{iKr}}{r} \right).
$$
By definition of the cut-off function, for any $k \in \N$, we have 
\begin{equation}
f^{(k)}({r_c} - \eta) = 0,
\end{equation}
and for $k \in \N^*$, $(1-\omega)^{(k)}({r_c}) = 0$. Thus by integration by parts, 
$$
\itg{{r_c}-\eta}{{r_c}}{f(r) e^{iKr}}{r} = \left[ f(r) \frac{e^{iKr}}{iK} \right]_{{r_c}-\eta}^{r_c} - \frac{1}{iK} \itg{{r_c}-\eta}{{r_c}}{f'(r) e^{iKr}}{r} = 
\frac{i}{K} \itg{{r_c}-\eta}{{r_c}}{f'(r) e^{iKr}}{r}.
$$
As $\cR-\widetilde{\cR}$ is $C^{d-1}$ but not $C^{d}$ at ${r_c}$, by integrating by parts $d$ times, we have:
\begin{equation}
\itg{{r_c}-\eta}{{r_c}}{f(r) e^{iKr}}{r} = \frac{i^{d+1} {r_c}}{K^{d+1}} [\widetilde{\cR}^{(d)}]_{{r_c}}^{} e^{iK{r_c}} - \frac{i^{d+1} }{K^{d+1}} \itg{{r_c} - \eta}{{r_c}}{f^{(d+1)}(r) e^{iKr}}{r}.
\end{equation}
Thus inserting the last equation in \eqref{eq:dth-der1}, we obtain
\begin{multline*}
\itg{[-\frac{1}{2},\frac{1}{2}]^3}{}{(1-\omega(r))\psh{\widetilde{p}}{\widetilde{\psi}}^T (\Phi - \widetilde{\Phi})  e^{-i \mathbf{K}\cdot \mathbf{r}}}{\mathbf{r}} = \mathrm{Im}(i^{d+1}e^{iK{r_c}}) \frac{4\pi{r_c}}{K^{d+2}} \psh{\widetilde{p}}{\widetilde{\psi}}^T [\widetilde{\cR}^{(d)}]_{{r_c}}^{} \\
- \mathrm{Im} \left(\frac{4 \pi i^{d+1} }{K^{d+2}} \itg{{r_c} - \eta}{{r_c}}{\psh{\widetilde{p}}{f}^T f^{(d+1)}(r) e^{iKr}}{r}\right).
\end{multline*}
According to Lemma \ref{lem:d-th der jump}, 
\[
\left| \psh{\widetilde{p}}{\widetilde{\psi}}^T [\widetilde{\cR}^{(d)}]_{{r_c}} \right| \leq  \frac{C}{{r_c}^{d-1}}.
\]
Furthermore, we have
\begin{equation*}
\psh{\widetilde{p}}{\widetilde{\psi}}^T f^{(d+1)}(r) = r \left( (1-\omega) \psh{\widetilde{p}}{\widetilde{\psi}} (\cR - \widetilde{\cR}) \right)^{(d+1)} + \left( (1-\omega) \psh{\widetilde{p}}{\widetilde{\psi}} (\cR - \widetilde{\cR}) \right)^{(d)} .
\end{equation*}
By assumption on $\omega$, we know that $\|\omega^{(k)}\|_{L^\infty(0,r_c)} \leq C {r_c}^{-k}$, hence
\begin{align*}
\left\| \psh{\widetilde{p}}{\widetilde{\psi}}^T f^{(d+1)} \right\|_{L^\infty(0,r_c)} & \leq r_c \sumlim{k=0}{d+1} \binom{d+1}{k} \left\| (1-\omega)^{(k)} (\cR - \widetilde{\cR})^{(d+1-k)} \right\|_{L^\infty(0,r_c)} \\
& \qquad \qquad + \sumlim{k=0}{d} \binom{d}{k} \left\| (1-\omega)^{(k)} (\cR - \widetilde{\cR})^{(d-k)} \right\|_{L^\infty(0,r_c)} \\
& \leq \frac{C}{{r_c}^d},
\end{align*}
where we used Lemma \ref{lem:d-th der jump}. Thus, we obtain
\begin{align*}
\left| \itg{{r_c} - \eta}{{r_c}}{\psh{\widetilde{p}}{f}^T f^{(d+1)}(r) e^{iKr}}{r} \right| & \leq r_c \left\| \psh{\widetilde{p}}{\widetilde{\psi}}^T f^{(d+1)} \right\|_{L^\infty(0,r_c)} \\ 
& \leq \frac{C}{{r_c}^{d-1}},
\end{align*}
which finishes the proof of this proposition.
\end{proof}

\subsection{Convergence theorem}

To prove the estimate on the eigenvalues, we will use the following classical result (\cite{weinberger1974variational}, p. 68).

\begin{prop}
\label{prop:eig_error}
Let $H$ be a self-adjoint coercive $H^1$-bounded operator, $E_1 \leq \dots \leq E_n$ be the lowest eigenvalues of $H$ and $\psi_1, \dots, \psi_n$ be $L^2$-normalized associated eigenfunctions.  Let $E^{(M)}_1 \leq \dots \leq E^{(M)}_n$ be the lowest eigenvalues of the Rayleigh quotient of $H$ restricted to the subspace $V_M$ of dimension $M$. 

Let $w_k \in V_M$ for $1 \leq k \leq n$ be such that 
\begin{equation*}
\sumlim{k=1}{n} \| w_k - \psi_k \|^2_{H^1} < 1
\end{equation*}

Then there exists a positive constant $C$ which depends on the $H^1$ norm of $H$ and the coercivity constant such that for all $1 \leq k \leq n$
\begin{equation*}
\left| E^{(M)}_k - E_k \right| \leq C \sumlim{k=1}{n} \| w_k - \psi_k \|^2_{H^1} 
\end{equation*}
\end{prop}

We would like to apply this result to $\psi_{K} = (\mathrm{Id}+T) \widetilde{\psi}_{K}$ where $\widetilde{\psi}_{K}$ is the truncation of the plane-wave expansion of $f$ to the wave number ${K}$. In order to do this, we need to show that the $H^1_\mathrm{per}$-norm of $(\mathrm{Id}+T)\widetilde{\psi}$ is bounded by a the $H^1_\mathrm{per}$-norm of $\widetilde{\psi}$ independently of the cut-off radius $r_c$.

\begin{lem}
\label{lem:Tcont}
There exists a positive constant $C$ independent of $r_c$ such that for any function $f \in H^1_\mathrm{per}([-\frac{1}{2},\frac{1}{2}]^3)$
$$
\left\| (\mathrm{Id}+T) f \right\|_{H^1_\mathrm{per}} \leq C \|f\|_{H^1_\mathrm{per}} .
$$
\end{lem}

\begin{proof}
By definition, we have $(\mathrm{Id}+T)f  = f + \psh{\widetilde{p}}{f}^T (\Phi-\widetilde{\Phi}) $.
By Lemma \ref{lem:norm_T}, we have 
$$
\psh{\widetilde{p}}{f}^T (\cR(r) - \widetilde{\cR}(r)) = C_{r_c}^T (C_{r_c} G(P) C_{r_c}^T)^{-1} C_{r_c} \itg{0}{1}{\chi(t) P(t) f_{00}(r_ct) t^2}{t} \cdot \left( \begin{pmatrix}
C_1^{-1} \\ 0
\end{pmatrix} \cR(r) - P(\tfrac{r}{r_c}) \right),
$$
with $\left\|\begin{pmatrix} C_1^{-1} \\ 0\end{pmatrix} \cR(r)\right\|_{L^\infty(0,r_c)}$ uniformly bounded with respect to $r_c$. Thus it suffices to bound $\itg{0}{1}{t^2 \chi(r_c t) f_{00}(r_ct) P(t)}{t}$:
\begin{align*}
\left| \itg{0}{1}{t^2 \chi(r_c t) f_{00}(r_ct) P(t)^2}{t} \right| & \leq \left( \itg{0}{1}{t^2 \chi(r_ct)^2 P(t)}{t} \right)^{1/2} \left( \itg{0}{1}{t^2 f_{00}(r_ct)^2}{t} \right)^{1/2} \\
& \leq \frac{C}{{r_c}^{3/2}} \left( \itg{0}{r_c}{r^2 f_{00}(r)^2}{r} \right)^{1/2} \\
& \leq \frac{C}{{r_c}^{3/2}} \left( \itg{0}{r_c}{r^2 f_{00}(r)^6}{r} \right)^{1/6} \left( \itg{0}{r_c}{r^2}{r} \right)^{1/3} \\
& \leq \frac{C}{{r_c}^{1/2}} \|f\|_{L^6(B_{r_c})} \\
& \leq \frac{C}{{r_c}^{1/2}} \|f\|_{H^1(B_{r_c})}.
\end{align*}
We obtain  
$$
\left| \psh{p}{f_{00}}^T_{[0,r_c]} (\cR(r) - \widetilde{\cR}(r)) \right| \leq \frac{C}{{r_c}^{1/2}} \|f\|_{H^1(B_{r_c})}.
$$
Since $\left\|\begin{pmatrix} C_1^{-1} \\ 0\end{pmatrix} \cR'(r)\right\|_{L^\infty(0,r_c)} = \cO \left( \frac{1}{r_c} \right)$, we can prove similarly
$$
\left| \psh{p}{f_{00}}^T_{[0,r_c]} (\cR'(r) - \widetilde{\cR}'(r)) \right| \leq \frac{C}{{r_c}^{3/2}} \|f\|_{H^1(B_{r_c})}.
$$
Thus we get 
$$
\left\| (\mathrm{Id}+T) f \right\|_{H^1_\mathrm{per}} \leq C \|f\|_{H^1_\mathrm{per}}.
$$
\end{proof}

We have all the elements to prove the main convergence theorem.

\begin{proof}[Proof of Theorem \ref{theo:energy}]
Using Equation \eqref{eq:psi_decomposition}, Propositions \ref{prop:cusp}, \ref{prop:d-th derivative} and \ref{prop:remainder}, we have
\begin{align*}
\|\widetilde{\psi}_{M} - \widetilde{\psi}\|_{H^1_\mathrm{per}}  & \leq \left\| \widetilde{\psi}_M - \widetilde{\psi} - \eta_{M} + \eta\right\|_{H^1_\mathrm{per}} + \left\| \eta_M - \eta \right\|_{H^1_\mathrm{per}} \\
& \leq \left(\sumlim{K \geq M, \mathbf{K} \in (2\pi \Z)^3}{} (1+K^2) \left( \frac{{r_c}^{\min(n,3)+1-\varepsilon}}{K^4} + \frac{1}{{r_c}^{d-1}K^{d+2}} \right)^2 \right)^{\frac{1}{2}} + o \left( \frac{1}{M^{5/2 - \varepsilon}}\right). 
\end{align*}
By Proposition \ref{prop:eig_error}, we obtain
\begin{align*}
|E_M - E| & \leq C  \| (\mathrm{Id}+T)( \widetilde{\psi}_M - \widetilde{\psi}) \|_{H^1_\mathrm{per}}^2  \\
& \leq C  \| \widetilde{\psi}_M - \widetilde{\psi} \|_{H^1_\mathrm{per}}^2  \\
& \leq C \left( \left( \sumlim{K \geq M, \mathbf{K} \in (2\pi \Z)^3}{} (1+K^2) \left( \frac{{{r_c}}^{\min(2n,5)-\varepsilon}}{K^4} + \frac{1}{{{r_c}}^{d-1} K^{d+2}} \right)^2 \right)^{1/2} + o \left(\frac{1}{M^{5/2-\varepsilon}} \right) \right)^2    \\
& \leq C \left( \frac{{{r_c}}^{2\min(2n,5)-2\varepsilon}}{M^3} +\frac{{{r_c}}^{\min(2n,5)-\varepsilon}}{M^{4-\varepsilon}} + \frac{1}{{{r_c}}^{2d-2}M^{2d-1}} + o \left( \frac{1}{M^{5-\varepsilon}}\right) \right).
\end{align*}
\end{proof}

\section{Acknowledgements}

The author is grateful to Xavier Blanc, Antoine Levitt, Eric Cancès and Marc Torrent for helpful discussion and fruitful comments.

\appendix
\section{Appendix}

We have gathered in this section proofs of some technical lemmas, most of which are simple transpositions of lemmas that can be found in \cite{blanc2017paw,blanc2017vpaw1d}. 

\subsection{Results related to the weighted Sobolev space $\cK^{\infty,a}(\Gamma)$}

\begin{lem}
\label{lem:cK_a_decay}
Let $f \in \cK^{\infty,a}(\Gamma)$ and $0 < R <1$. Let $\ell \in \N$ and $m \in \N$ such that $|m| \leq \ell$. Then there exists a constant $C$ independent of $f$ such that
\[
\itg{0}{R}{|f_{\ell m}(r)| r^2}{r} \leq C R^{a+\frac{3}{2}} \|f\|_{\cK^{\infty,a}},
\]
and for $a \geq 1$
\[
\itg{0}{R}{|f_{\ell m}(r)|^2}{r} \leq C R^{2a-2} \|f\|_{\cK^{\infty,a}}.
\]
\end{lem}

\begin{proof}
Since $Y_{\ell m} \in L^\infty(S(0,1))$, we have
\begin{align*}
\itg{0}{R}{|f_{\ell m}(r)| r^2}{r} & \leq C \itg{B(0,R)}{}{|f(\mathbf{r})|}{\mathbf{r}} \\
& \leq C \itg{B(0,R)}{}{r^{a}  r^{-a}|f(\mathbf{r})|}{\mathbf{r}} \\
& \leq C \left( \itg{0}{R}{r^{2a+2}  }{r} \right)^{1/2} \left( \itg{B(0,R)}{}{r^{-2a} |f(\mathbf{r})|^2}{\mathbf{r}} \right)^{1/2} \\
& \leq C R^{a+\frac{3}{2}} \|f\|_{\cK^{\infty,a}},
\end{align*}
where in the fourth inequality we used the definition of the weighted Sobolev space $\cK^{\infty,a}$.
The second identity is proved the same way.
\end{proof}

\begin{lem}
\label{lem:size_weighted_sobolev_remainder}
Let $N \in \N^*$ and $\eta$ a radial function such that $\eta \in \cK^{\infty,5/2 + N -\varepsilon}(\Gamma)$, $\varepsilon>0$. Then for $R$ sufficiently small, we have
\[
\| \eta \|_{L^\infty(0,R)} \leq \|f\|_{\cK^{\infty,a}} R^{N+\frac{1}{2}-\varepsilon}.
\]
\end{lem}

\begin{proof}
By definition of the weighted Sobolev space, we have for $R$ sufficiently small,
\begin{equation*}
\itg{B(0,R)}{}{|\eta(r)|^2 r^{-5-2N+2\varepsilon}}{\mathbf{r}} < \infty,
\end{equation*}
hence
\[
\itg{0}{R}{|\eta(r)|^2 r^{-3-2N+2\varepsilon}}{r} < \infty,
\]
Similarly we have
\[
\itg{0}{R}{|\eta'(r)|^2 r^{-1-2N+2\varepsilon}}{r} < \infty.
\]
Therefore,
\[
\begin{cases}
\itg{0}{R}{|\eta(r)|^2}{r} \leq R^{2N+3-2\varepsilon} \itg{0}{R}{|\eta(r)|^2 r^{-1-2N+2\varepsilon}}{r}, \\
\itg{0}{R}{|\eta'(r)|^2}{r} \leq R^{2N+1-2\varepsilon} \itg{0}{R}{|\eta'(r)|^2 r^{-1-2N+2\varepsilon}}{r}.
\end{cases}
\]
By the Sobolev embedding theorem, we have the result. 
\end{proof}

\begin{note}
Lemma \ref{lem:size_weighted_sobolev_remainder} implies that the remainder $\eta_N$ of the singularity expansion \eqref{eq:weighted_sobolev} of radial functions are bounded : $\|\eta_N\|_{L^\infty(0,r_c)} \leq C {r_c}^{N+1}$, where the constant is independent of $r_c$. 
\end{note}

\subsection{Validity of the assumptions for the hydrogenoid atom}

We show in this subsection that the Assumptions \ref{assump:cR_free_family} and \ref{assump:radial_PAW_function_cusp} hold in the particular case of the hydrogenoid atom, \emph{i.e.} where in \eqref{eq:atomic_hamiltonian} $W_\mathrm{at}=0$. 
The eigenfunctions of the hydrogenoid atom can be written 
\[
\varphi_{n\ell m}(\mathbf{r}) = R_{n\ell}(r) Y_{\ell m}(\hat{\mathbf{r}}), \quad n \geq \ell+1,
\]
with
\begin{equation}
\label{eq:radial_hydrogenoid_wf}
R_{n\ell} (r) = \sqrt {{\left ( \frac{2 Z}{n} \right ) }^3\frac{(n-\ell-1)!}{2n(n+\ell)!} } e^{- Z r / {n}} \left ( \frac{2 Z r}{n } \right )^{\ell} L_{n-\ell-1}^{(2\ell+1)} \left ( \frac{2 Z r}{n } \right ),
\end{equation}
where $L_{n-\ell-1}^{(2\ell+1)}$ denotes the generalized Laguerre polynomials. The constant term of $L_{n-\ell-1}^{(2\ell+1)}$ is equal to $\binom{n + \ell}{n-\ell-1} > 0$, hence Assumption \ref{assump:radial_PAW_function_cusp} holds for the hydrogenoid atom wave function.

\begin{lem}
\label{lem:assumption_der_cR}
Let $R_{n\ell}(r) = L_{n-1}(\frac{2Z}{n}) e^{-\frac{Zr}{n}}$, where $\mathrm{deg}(L_{n-1}) = n-1$. Let $r>0$ and $\cR_k = (R_{1\ell}^{(k)}(r), \dots, R_{n\ell}^{(k)}(r))^T$ for $0 \leq k \leq n-1$. Then the matrix $(\cR_0, \dots, \cR_{n-1})$ is invertible. 
\end{lem}

\begin{proof}
 We have
$$
R_{n\ell}^{(k)}(r) = \left( \frac{Z}{n} \right)^k e^{-\frac{Zr}{n}} \sumlim{j=0}{k} 2^j \tbinom{k}{j} L_{n-1}^{(j)}\left( \frac{2Zr}{n} \right).
$$
Let $\mathcal{P} = \left(L_{j}^{(k)}(\frac{2Zr}{j+1}) \right)_{0 \leq j,k \leq n-1}$, $\mathcal{M} = (2^j \binom{k}{j})_{0 \leq j,k \leq n-1}$ and $\mathcal{Z} = \left((\frac{Z}{k+1})^j\right)_{0 \leq j,k \leq n-1}$. One can check that
$$
(\cR_0, \dots, \cR_{n-1})=\mathcal{Z} \ \mathrm{diag}(e^{-Zr}, \dots, e^{-\frac{Zr}{n}}) \mathcal{P} \mathcal{M}.
$$
$\mathcal{P}$ and $\mathcal{M}$ are both triangular with no null entry on the diagonal. $\mathcal{Z}$ is a Vandermonde matrix, hence $(\cR_0, \dots, \cR_{n-1})$ is invertible. 
\end{proof}

\subsection{Lemmas related to PAW functions}

Let $P_k$, $k\in \N$ the polynomials defined by 
\begin{equation}
\label{eq:P_k}
P_k(t) = \frac{1}{2^k k!}(t^2-1)^k.
\end{equation} 
By definition, these polynomials form a basis of even polynomials and satisfy
$$
\begin{cases}
P_k^{(j)}(1) = 0, \quad 0 \leq j \leq k-1,\\
P_k^{(k)}(1) = 1.
\end{cases}
$$
Let $P$ be the vector $(P_0,\dots,P_{d-1})^T$ and $C_{r_c} \in \R^{n \times d}$ be the matrix such that 
\begin{equation}
\label{eq:C_r_c}
\widetilde{\cR}(t) = C_{r_c} P(\tfrac{t}{r_c}).
\end{equation}
The following lemma summarizes the main properties of the matrix $C_{r_c}$.

\begin{lem}
\label{lem:C_1C2}
Let $C_1 \in \R^{n \times n}$ and $C_2 \in R^{n \times (n-d)}$ be the matrices such that 
$$
\widetilde{\mathcal{R}}(r) = \Big( C_1 \ \Big| \ C_2 \Big) P(\tfrac{r}{r_c}).
$$
Moreover, $C_1$ is invertible, the norm of $C_1^{-1} C_2$ is uniformly bounded with respect to $r_c$ and 
$$
C_2^T C_1^{-T} e_0 = \cO(r_c).
$$
\end{lem}

\begin{proof}
Let $C_{r_c}$ be the matrix $ \Big( C_1 \ \Big| \ C_2 \Big)$. Let $c_j$ be the columns of $C_{r_c}$. By continuity of $\widetilde{\cR}$ and of its derivatives at $r_c$, and by our choice of the polynomials $P_k$, the columns of $C_{r_c}$ satisfy
\begin{equation}
\label{eq:recurrence_colonne}
\forall \, 0 \leq j \leq d-1, \ c_j = {r_c}^j \cR^{(j)}(r_c) - \sumlim{k=0}{j-1} P_k^{(j)}c_k. 
\end{equation}
Hence $c_k$ is a linear combination of the vectors ${r_c}^j \cR^{(j)}(r_c)$ for $j \leq k$ with coefficients that are \emph{independent} of $r_c$. Moreover we can deduce that the transformation of $(c_j)_{0\leq j \leq n-1}$ to $({r_c}^j \cR^{(j)}(r_c))_{0\leq j \leq n-1}$ is invertible. 
If $r_c$ is sufficiently small, by Assumption \ref{assump:cR_free_family}, the family $({r_c}^j \cR^{(j)}(r_c))_{0\leq j \leq n-1}$ is linearly independent, thus we can define $(g_j)_{0\leq j \leq n-1}$ to be the dual family to $(c_j)_{0\leq j \leq n-1}$ (\emph{i.e.} $c_j^T g_k = \delta_{jk}$) and we have $\|g_j\| = \cO \left( \tfrac{1}{{r_c}^{n-1}} \right)$. 
Hence, using the recurrence \eqref{eq:recurrence_colonne}, we can show that the norm of $C^{-1}_1 C_2$ is uniformly bounded with respect to $r_c$. 

To prove $C_2^T C_1^{-T} e_0 = \cO(r_c)$, first notice that $C_1^{-T}e_0 = g_1$. Since $P_0$ is a constant polynomial, for $j \geq 1$, we have
\begin{equation*}
c_j = {r_c}^j \cR^{(j)}(r_c) - \sumlim{k=1}{j-1} P_k^{(j)}c_k. 
\end{equation*}
Thus for $j = n$, we have
$$
c_n^T g_0 = {r_c}^n \cR^{(n)}(r_c)^T g_0 - \sumlim{k=1}{j-1} P_k^{(j)}c_k^T g_0 = \cO(r_c),
$$
and by iteration, we can check that 
$$
\forall \, n \leq j \leq d-1, \ c_j^T g_0 = \cO(r_c).
$$
\end{proof}

\begin{lem}
\label{lem:norm_T}
We have
$$
\psh{\widetilde{p}}{f}^T (\cR(r) - \widetilde{\cR}(r)) = C_{r_c}^T (C_{r_c} G(P) C_{r_c}^T)^{-1} C_{r_c} \itg{0}{1}{\chi(t) P(t) f_{00}(r_ct) t^2}{t} \cdot \left( \begin{pmatrix}
C_1^{-1} \\ 0
\end{pmatrix} \cR(r) - P(\tfrac{r}{r_c}) \right),
$$
with $P$ being the vector of the polynomials $P_k$ defined in \eqref{eq:P_k}, $C_{r_c}$ the matrix of coefficients of $\widetilde{\cR}$ in the basis $(P_k)$ given in \eqref{eq:C_r_c} and $G(P)$ the matrix 
\[
G(P) = \itg{0}{1}{\chi(t) P(t) P(t)^T t^2}{t}.
\]
The norm of the matrix $C_{r_c}^T (C_{r_c} G(P) C_{r_c}^T)^{-1} C_{r_c}$ is uniformly bounded as $r_c$ goes to 0 and we have
\[
\left\| \begin{pmatrix}
C_1^{-1} \\ 0
\end{pmatrix} \cR(r) \right\|_{L^\infty(0,r_c)} \leq C \quad \text{ and } \quad \left\| \begin{pmatrix}
C_1^{-1} \\ 0
\end{pmatrix} \cR'(r) \right\|_{L^\infty(0,r_c)} \leq \frac{C}{r_c}.
\]
\end{lem}

\begin{proof}
Since $G(P)$ is positive-definite, $G(P)^{1/2}$ exists. Note that the matrix \linebreak $G(P)^{1/2} C_{r_c}^T (C_{r_c} G(P) C_{r_c}^T)^{-1} C_{r_c} G(P)^{1/2}$ is symmetric and is a projector, hence its norm is independent of $r_c$. 

Writing down the Taylor expansion of $\cR$ at $r_c$, we obtain
\begin{align*}
\cR(r) & =  \sumlim{k=0}{n-1} \frac{(r-r_c)^k}{k!} \cR^{(k)}(r_c) + \cO((r-r_c)^n) \\
& =  \sumlim{k=0}{n-1} \frac{1}{k!} \left( \frac{r}{r_c}-1 \right)^k {r_c}^k \cR^{(k)}(r_c) + \cO((r-r_c)^n). 
\end{align*}
By Lemma \ref{lem:C_1C2}, $\| C_1^{-1} {r_c}^k \cR^{(k)}(r_c)\|$ is uniformly bounded as $r_c$ goes to 0, which concludes the proof of the lemma.
\end{proof}

\begin{lem}
\label{lem:tilde_p_stuff}
Let $f \in L^2(\Gamma)$. Let $\ell \geq 0$, $|m| \leq \ell$ be integers. Let $n$ be the number of PAW functions associated to the angular momentum $\ell,m$ for a cut-off radius $r_c$. There exists a constant independent of $r_c$ and $f$ such that
$$
|\psh{\widetilde{p}}{f}^T \cR(0)| \leq \frac{C}{{r_c}^{3/2}} \|f\|_{L^2(B_{r_c})}.
$$
\end{lem}

\begin{proof}
The proof of this lemma is similar to the proof of Proposition 3.2 in \cite{blanc2017vpaw1d}. 
We give the proof in case $\ell=m =0$. First, it is possible to show that 
\begin{equation}
\label{eq:tildep_f}
\psh{\widetilde{p}}{f} = \itg{0}{1}{\chi(t) (C_{r_c}^{(Q)} G_{r_c})^{-1} C_{r_c}^{(Q)} Q(t) f(r_c t) t^2}{t},
\end{equation}
where 
$Q(t) = (Q_0(t),\dots,Q_{d-1}(t))^T$ is a vector of even polynomials which forms a basis of even polynomials of degree at most $2d-2$,
$$
\widetilde{\cR}(x) = C_{r_c}^{(Q)} Q( \tfrac{x}{r_c} ), 
$$
with $C_{r_c}^{(Q)} \in \R^{n \times d}$ and
$$
G_{r_c} = \itg{0}{1}{\chi(t)  Q(t) \mathcal{R}(r_c t)^T t^2}{t} \in \R^{d \times n}.
$$
By Lemma \ref{lem:zeta_k}, we have that $(\zeta_{2k})_{0 \leq k \leq n-1}$ and $(\zeta_{2k+1})_{0 \leq k \leq n-1}$ defined by the singular expansion of $\cR$ :
\begin{equation}
\label{eq:cR_singular_expansion}
\mathcal{R}(t) = \sumlim{k=0}{n-1} \zeta_{2k} t^{2k} + \zeta_{2k+1} t^{2k+1} + \eta_{2n}(t),
\end{equation}
satisfy
\begin{equation*}
\zeta_{2k}   = \sumlim{j=0}{k} \mu_j^{(k)} \cE^j \zeta_0 \quad \text{ and } \quad \zeta_{2k+1}   = \sumlim{j=0}{k} \nu_j^{(k)} \cE^j \zeta_0 ,
\end{equation*}
where $\mu_k^{(k)} \not= 0$ and $\cE$ is the diagonal matrix of the eigenvalues $(\epsilon_{1}, \dots, \epsilon_n)$. By Assumption \ref{assump:radial_PAW_function_cusp}, $\zeta_0$ has no null entry. The eigenvalues of the atomic operator \eqref{eq:atomic_hamiltonian} for a fixed $\ell,m$ are simple, hence $(\cE^j \zeta_0)_{0 \leq j \leq n-1}$ is a linearly independent family. Hence, $(\zeta_{2k})_{0 \leq k \leq n-1}$ is a basis of $\R^n$. Let $(h_k)$ be the dual basis to $({r_c}^{2j} \cE^j \zeta_0)_{0 \leq j \leq n-1}$, \emph{i.e.} $h_k^T {r_c}^{2j} \cE^j \zeta_0 = \delta_{kj}$. 

Injecting \eqref{eq:cR_singular_expansion} in the definition of $G_{r_c}$, we obtain
\begin{align}
G_{r_c}^T & = \itg{0}{1}{\chi(t) \cR(r_c t) Q(t)^T t^2}{t} \\
& = \itg{0}{1}{ \chi(t) \left( \sumlim{k=0}{n-1} {r_c}^{2k} \cE^k \zeta_0 \sumlim{j=k}{n-1} \mu_k^{(j)} {r_c}^{2j-2k} t^{2j} + r_c \sumlim{k=0}{n-1} {r_c}^{2k} \cE^k \zeta_0 \sumlim{j=k}{n-1} \nu_k^{(j)} {r_c}^{2j-2k} t^{2j+1} +\eta_{2n}(r_c t)  \right) Q(t)^T t^2}{t}.
\end{align}
By Lemma \ref{lem:cK_a_decay}, we have 
\begin{equation}
\left| \itg{0}{1}{\chi(t) \eta_{2n}(r_c t) Q(t)^T t^2}{t} \right| = \frac{1}{r_c^3} \left| \itg{0}{r_c}{\chi(\tfrac{t}{r_c}) \eta_{2n}(t) Q(\tfrac{t}{r_c})^T t^2}{t} \right| \leq C {r_c}^{2n}.
\end{equation}
Let $(Q_k)_{0 \leq k \leq d-1}$ be the even polynomials such that 
$$
\itg{0}{1}{\chi(t) t^{2j} Q_k(t) t^2}{t} = \delta_{jk}.
$$
Such polynomials exist since the Gram matrix $(\itg{0}{1}{\chi(t) t^{2j+2k+2}}{t})_{0\leq j,k \leq n-1}$ is invertible. Let 
$$
X_j = \itg{0}{1}{ \chi(t)  t^{2j+1} Q(t) t^2}{t},
$$
and 
\begin{equation}
\label{eq:mathcal_H}
\mathcal{H} = \begin{pmatrix}
h_0^T \\
\vdots \\
h_{n-1}^T
\end{pmatrix} \in \R^{n \times n}.
\end{equation}
Then denoting by $e_k$ the $k$-th canonical vector, we have
\begin{align*}
\mathcal{H}  G_{r_c}^T & = \sumlim{k=0}{n-1} \mu_k^{(k)} e_k e_k^T + \sumlim{k=0}{n-1} \sumlim{j=k+1}{n-1} \mu_k^{(j)} {r_c}^{2j-2k} e_k e_j^T + r_c \sumlim{k=0}{n-1} \nu_k^{(j)} {r_c}^{2j-2k} e_k X_j^T + \cO({r_c}^2) \\
& = \sumlim{k=0}{n-1} \mu_k^{(k)} e_k e_k^T + \cO({r_c}).
\end{align*}
Let 
\begin{equation}
\label{eq:mathcal_A}
\mathcal{A} = \sumlim{k=0}{n-1} \mu_k^{(k)} e_k e_k^T \in \R^{n \times n},
\end{equation}
and $\Pi$ the transition matrix such that
$$
C_{r_c}^{(Q)} = C_{r_c} \Pi,
$$
where $C_{r_c}$ is defined in \eqref{eq:C_r_c}. Hence we have 
\begin{align*}
(C_{r_c}^{(Q)})^T (C_{r_c}^{(Q)} G_{r_c})^{-T} \cR(0) &= (C_{r_c}^{(Q)})^T (C_{r_c}^{(Q)} G_{r_c})^{-T} \mathcal{H}^{-1} e_0 \\
& = (C_{r_c}^{(Q)})^T (C_{r_c}^{(Q)} G_{r_c} \mathcal{H}^T)^{-T} e_0 \\
& = \Pi^T \begin{pmatrix}
I_n \\ M^T + \cO(r_c)
\end{pmatrix} \left( C^{-1} C_{r_c} \Pi \left( \begin{pmatrix} \mathcal{A} \\ 0\end{pmatrix} + \cO(r_c)\right) \right)^{-T} e_0 \\
& =  \Pi^T \begin{pmatrix}
I_n \\ M^T + \cO(r_c)
\end{pmatrix} \left( \Big( \ I_n \Big| \ M + \cO(r_c) \Big) \Pi \left( \begin{pmatrix} \mathcal{A} \\ 0\end{pmatrix} + \cO(r_c) \right) \right)^{-T} e_0 ,
\end{align*}
where we used Lemma \ref{lem:C_1C2} in the third and fourth inequality. Decomposing $\Pi$ into four blocks 
\[
\Pi = \begin{pmatrix}
\Pi_1 & \Pi_2 \\
\Pi_3 & \Pi_4
\end{pmatrix}, \quad \text{with } \Pi_1 \in \R^{n \times n},
\]
we obtain 
\begin{align*}
(C_{r_c}^{(Q)})^T (C_{r_c}^{(Q)} G_{r_c})^{-T} \cR(0) &= \left( \begin{pmatrix}
\Pi_1^T + \Pi_3^T M^T \\
\Pi_2^T + \Pi_4^T M^T
\end{pmatrix} + \cO(r_c) \right) \left( \Pi_1 \mathcal{A} + M \Pi_3 \mathcal{A} + \cO(r_c) \right) e_0 \\
& = \begin{pmatrix}
\mathcal{A}^{-1} \\
(\Pi_2^T + \Pi_4^T M^T)(\Pi_1 + M \Pi_3)^{-1}
\end{pmatrix} e_0 + \cO(r_c).
\end{align*}
Hence $\| (C_{r_c}^{(Q)})^T (C_{r_c}^{(Q)} G_{r_c})^{-T} \cR(0)\|$ is uniformly bounded as $r_c$ goes to 0. Thus, there exists a constant $C$ independent of $r_c$ and $f$ such that :
\begin{align*}
\left| \psh{\widetilde{p}}{f}^T \cR(0) \right| & \leq C \itg{0}{1}{ |f(r_c t)| t^2}{t} \\
& \leq \frac{C}{{r_c}^{3/2}} \|f\|_{L^2_\mathrm{per}}.
\end{align*}
\end{proof}

We can now prove Lemma \ref{lem:d-th der jump}.

\begin{proof}[Proof of Lemma \ref{lem:d-th der jump}]
We start with the proof of the estimate of $\left[ \psh{\widetilde{p}}{\widetilde{\psi}}^T \widetilde{\cR}^{(d)}  \right]_{r_c}$.  We have using \eqref{eq:tildep_f} and \eqref{eq:C_r_c}
\begin{align}
\left[ \psh{\widetilde{p}}{\widetilde{\psi}}^T \widetilde{\cR}^{(d)}  \right]_{r_c} & = \psh{\widetilde{p}}{\widetilde{\psi}}^T (\cR^{(d)}(r_c) - \widetilde{\cR}^{(d)}(r_c)) \\ 
& = \frac{1}{{r_c}^d} \itg{0}{1}{\chi(t) \psi_{00}(r_c t)Q(t) t^2}{t} \cdot (C^{(Q)}_{r_c})^T (C^{(Q)}_{r_c} G_{r_c})^{-T} \left( {r_c}^d \cR^{(d)}(r_c) - C_{r_c} P^{(d)}(1)  \right) .
\end{align}

First, we prove that 
\begin{equation}
\label{eq:d-th-der_step1}
(C^{(Q)}_{r_c})^T (C^{(Q)}_{r_c} G_{r_c})^{-T} \left( {r_c}^d \cR^{(d)}(r_c) - C_{r_c} P^{(d)}(1) \right) = \begin{pmatrix}
0 \\
*
\end{pmatrix} + \cO(r_c),
\end{equation}
then 
\begin{equation}
\label{eq:d-th-der_step2}
\itg{0}{1}{\chi(t)\psi_{00}(r_ct) Q(t) t^2}{t} = \psi(0) e_0 + \cO(r_c).
\end{equation}

If both statements are true, then we deduce that there exists a constant $C$ independent of $r_c$ such that
$$
\left| \left[ \psh{\widetilde{p}}{\widetilde{\psi}}^T \widetilde{\cR}^{(d)}  \right]_{r_c} \right| \leq \frac{C}{{r_c}^{d-1}} .
$$

\paragraph{Step 1 (proof of \eqref{eq:d-th-der_step1})}

By \eqref{eq:cR_zeta_expansion}, we have for $0\leq j \leq 2n-1$ even
\begin{align*}
{r_c}^j \cR^{(j)}(r_c)& = \sumlim{k=j/2}{n-1} \zeta_{2k} \frac{(2k)!}{(2k-j)!} {r_c}^{2k} + \zeta_{2k+1} \frac{(2k+1)!}{(2k+1-j)!} {r_c}^{2k+1} + \cO({r_c}^{2n})\\ 
& = \sumlim{k=j/2}{n-1} \sumlim{\ell=0}{k} \mu_\ell^{(k)} {r_c}^{2\ell}\cE^\ell \zeta_0 \frac{(2k)!}{(2k-j)!} {r_c}^{2k-2\ell} +  r_c \sumlim{k=j/2}{n-1} \sumlim{\ell=0}{k} \nu_\ell^{(k)} {r_c}^{2\ell}\cE^\ell \zeta_0 \frac{(2k+1)!}{(2k+1-j)!} {r_c}^{2k-2\ell} + \cO({r_c}^{2n}),
\end{align*}
where we applied Lemma \ref{lem:size_weighted_sobolev_remainder} to estimate the remainder of the singular expansion.
By noticing that $\mathcal{H} {r_c}^{2j} \cE^j \zeta_0 = e_j$ with $\mathcal{H}$ defined in \eqref{eq:mathcal_H}, we have  
\[
(C^{(Q)}_{r_c})^T (C^{(Q)}_{r_c} G_{r_c})^{-T} {r_c}^j \cE^j \zeta_0 = \begin{pmatrix}
\mathcal{A}^{-1} \\ *
\end{pmatrix} e_j + \cO(r_c).
\]
Using $\|(C^{(Q)}_{r_c}) (C^{(Q)}_{r_c} G_{r_c})^{-1}\| = \cO \left(\frac{1}{{r_c}^{2n-2}}\right)$, we thus get 
\begin{align}
\label{eq:d-th-der-jump_even}
(C^{(Q)}_{r_c})^T (C^{(Q)}_{r_c} G_{r_c})^{-T} {r_c}^j \cR^{(j)}(r_c) & = \begin{pmatrix}
I_n \\ *
\end{pmatrix} \sumlim{k=j/2}{n-1}  e_k + \cO(r_c).
\end{align}

For $0\leq j \leq 2n-1$ odd, we have
$$
{r_c}^j \cR^{(j)}(r_c) = {r_c}^j \zeta_j + \sumlim{k=\frac{j+1}{2}}{n-1} \zeta_{2k} \frac{(2k)!}{(2k-j)!} {r_c}^{2k}+ \zeta_{2k+1} \frac{(2k+1)!}{(2k+1-j)!} {r_c}^{2k+1}+ \cO({r_c}^{2n}),
$$
similarly to the even case, we thus obtain 
\begin{align}
\label{eq:d-th-der-jump_odd}
(C^{(Q)}_{r_c})^T (C^{(Q)}_{r_c} G_{r_c})^{-T} {r_c}^j \cR^{(j)}(r_c) & = \begin{pmatrix}
\mathcal{A}^{-1} \\ *
\end{pmatrix} \sumlim{k=\frac{j+1}{2}}{n-1}  \nu_k^{(k)} e_k + \cO(r_c).
\end{align}

For $j \geq 2n$, using $\|(C^{(Q)}_{r_c}) (C^{(Q)}_{r_c} G_{r_c})^{-1}\| = \cO\left(\frac{1}{{r_c}^{2n-2}}\right)$, then 
\begin{equation}
\label{eq:d-th-der-jump_d_geq_2n}
(C^{(Q)}_{r_c})^T (C^{(Q)}_{r_c} G_{r_c})^{-T} {r_c}^j \cR^{(j)}(r_c) = \cO(r_c).
\end{equation}
From \eqref{eq:d-th-der-jump_even} (when $d \leq 2n-1$ and $d$ is even), \eqref{eq:d-th-der-jump_odd} (when $d \leq 2n-1$ and $d$ is odd) or \eqref{eq:d-th-der-jump_d_geq_2n} (when $d \geq 2n$), we have
$$
(C^{(Q)}_{r_c})^T (C^{(Q)}_{r_c} G_{r_c})^{-T} {r_c}^{d} \cR^{(d)}(r_c) = \begin{pmatrix}
0 \\
*
\end{pmatrix} + \cO(r_c).
$$
It remains to prove the same statement for the other part. By definition of the polynomials $P_k$ \eqref{eq:P_k}, we have
$$
P^{(d)}(1) = (\underbrace{0,\dots,0}_{\lfloor \frac{d}{2} \rfloor}, *, \dots,*)^T,
$$
so $C_{r_c} P^{(d)}(1) $ is a linear combination of the last $\lceil \frac{d}{2} \rceil$ columns of $C_{r_c}$. However, by Lemma \ref{lem:C_1C2}, we know that except the first column of $C_{r_c}$, the columns of $C_{r_c}$ do not depend on $\cR(r_c)$ and by \eqref{eq:d-th-der-jump_even} and \eqref{eq:d-th-der-jump_odd}, for $j \geq 1$, 
$$
e_0^T (C^{(Q)}_{r_c})^T (C^{(Q)}_{r_c} G_{r_c})^{-T} {r_c}^j \cR^{(j)}(r_c) = \cO(r_c),
$$
which finishes the proof of the first step. 

\paragraph{Step 2 (proof of \eqref{eq:d-th-der_step2})}

Since $\psi \in H^2_\mathrm{per}(\Gamma)$, by Sobolev embedding theorem, $\psi$ is continuous, hence $\psi(0)$ is finite. Thus
\begin{align*}
 \itg{0}{1}{\chi(t)\psi_{00}(r_ct) Q(t) t^2}{t}  & =  \itg{0}{1}{\chi(t)\left( \psi(0) + \itg{0}{r_ct}{ \psi_{00}'(u)}{u} \right) Q(t) t^2}{t}  \\
& = \psi(0) e_0 + \itg{0}{1}{\chi(t)  \itg{0}{r_ct}{ \psi_{00}'(u)}{u} \, Q(t) t^2}{t},
\end{align*}
by definition of the polynomials $Q_k$. 

We have ($C$ denotes a constant independent of $r_c$)
\begin{align*}
\left| \itg{0}{1}{\chi(t)  \itg{0}{r_ct}{ \psi_{00}'(u)}{u} \ Q(t) t^2}{t} \right| & \leq \left( \itg{0}{1}{\chi(t)^2   Q(t)^2 t^4}{t} \itg{0}{1}{ \left(\itg{0}{r_ct}{ \psi_{00}'(u)}{u} \right)^2 }{t} \right)^{1/2} \\
& \leq C r_c \left( \itg{0}{1}{ \left( \frac{1}{u} \psi_{00}' \right)^2 u^2}{u} \right)^{1/2} .
\end{align*}
Using Hardy inequality, we get
\begin{align*}
\left| \itg{0}{1}{\chi(t)  \itg{0}{r_ct}{ \psi_{00}'(u)}{u} \  Q(t) t^2}{t} \right| & \leq C r_c \|\psi\|_{H^2_\mathrm{per}}, 
\end{align*}
which ends the proof of \eqref{eq:d-th-der_step2}.
\newline

The proof of the bound on $\psh{\widetilde{p}}{\widetilde{\psi}}^T (\cR^{(k)}-\widetilde{\cR}^{(k)})$ is a direct extension of the proof of \eqref{eq:d-th-der_step1}.
\end{proof}

\bibliographystyle{siam}
\bibliography{VPAW-3D}

\begin{thebibliography}{10}

\bibitem{bezanson2017julia}
{\sc J.~Bezanson, A.~Edelman, S.~Karpinski, and V.~B. Shah}, {\em Julia: A
  fresh approach to numerical computing}, SIAM review, 59 (2017), pp.~65--98.

\bibitem{blanc2017}
{\sc X.~Blanc, E.~Canc\`es, and M.-S. Dupuy}, {\em Variational projector
  augmented-wave method}, Comptes Rendus Mathematique, 355 (2017), pp.~665 --
  670.

\bibitem{blanc2017vpaw1d}
{\sc X.~Blanc, E.~Canc{\`e}s, and M.-S. Dupuy}, {\em Variational projector
  augmented-wave method: theoretical analysis and preliminary numerical
  results}, Numer. Math., 144 (2020), pp.~271--321.

\bibitem{blochl94}
{\sc P.~E. Blochl}, {\em Projector augmented-wave method}, Phys. Rev. B, 50
  (1994), pp.~17953--17979.

\bibitem{cances2014mathematical}
{\sc E.~Canc\`es and N.~Mourad}, {\em A mathematical perspective on density
  functional perturbation theory}, Nonlinearity, 27 (2014), pp.~1999--2033.

\bibitem{cances2016existence}
{\sc E.~Cances and N.~Mourad}, {\em Existence of a type of optimal
  norm-conserving pseudopotentials for {K}ohn--{S}ham models}, Communications
  in Mathematical Sciences, 14 (2016), pp.~1315--1352.

\bibitem{catto2001thermodynamic}
{\sc I.~Catto, C.~Le~Bris, and P.-L. Lions}, {\em On the thermodynamic limit
  for {H}artree-{F}ock type models}, vol.~18, 2001, pp.~687--760.

\bibitem{catto2002some}
\leavevmode\vrule height 2pt depth -1.6pt width 23pt, {\em On some periodic
  {H}artree-type models for crystals}, vol.~19, 2002, pp.~143--190.

\bibitem{chen2015numerical}
{\sc H.~Chen and R.~Schneider}, {\em Numerical analysis of augmented plane wave
  methods for full-potential electronic structure calculations}, ESAIM Math.
  Model. Numer. Anal., 49 (2015), pp.~755--785.

\bibitem{dupuy2018analysis}
{\sc M.-S. Dupuy}, {\em {Analysis of the projector augmented-wave method for
  electronic structure calculations in periodic settings}}, theses,
  {Universit{\'e} Sorbonne Paris Cit{\'e}}, Sept. 2018.

\bibitem{blanc2017paw}
\leavevmode\vrule height 2pt depth -1.6pt width 23pt, {\em Projector
  augmented-wave method: an analysis in a one-dimensional setting}, ESAIM Math.
  Model. Numer. Anal., 54 (2020), pp.~25--58.

\bibitem{eastham1973spectral}
{\sc M.~S.~P. Eastham}, {\em The spectral theory of periodic differential
  equations}, Texts in Mathematics (Edinburgh), Scottish Academic Press,
  Edinburgh; Hafner Press, New York, 1973.

\bibitem{egorov2012pseudo}
{\sc Y.~V. Egorov and B.-W. Schulze}, {\em Pseudo-differential operators,
  singularities, applications}, vol.~93 of Operator Theory: Advances and
  Applications, Birkh\"{a}user Verlag, Basel, 1997.

\bibitem{flad2008asymptotic}
{\sc H.-J. Flad, R.~Schneider, and B.-W. Schulze}, {\em Asymptotic regularity
  of solutions to {H}artree-{F}ock equations with {C}oulomb potential}, Math.
  Methods Appl. Sci., 31 (2008), pp.~2172--2201.

\bibitem{gontier2015mathematical}
{\sc D.~Gontier}, {\em Mathematical contributions to the calculations of
  electronic structures}, PhD thesis, Universit{\'e} Paris-Est, 2015.

\bibitem{hunsicker2008analysis}
{\sc E.~Hunsicker, V.~Nistor, and J.~O. Sofo}, {\em Analysis of periodic
  {S}chr\"{o}dinger operators: regularity and approximation of eigenfunctions},
  J. Math. Phys., 49 (2008), pp.~083501, 21.

\bibitem{jollet2014generation}
{\sc F.~Jollet, M.~Torrent, and N.~Holzwarth}, {\em Generation of {P}rojector
  {A}ugmented-{W}ave atomic data: {A} 71 element validated table in the {XML}
  format}, Computer Physics Communications, 185 (2014), pp.~1246--1254.

\bibitem{kato1957eigenfunctions}
{\sc T.~Kato}, {\em On the eigenfunctions of many-particle systems in quantum
  mechanics}, Comm. Pure Appl. Math., 10 (1957), pp.~151--177.

\bibitem{koelling1975use}
{\sc D.~Koelling and G.~Arbman}, {\em Use of energy derivative of the radial
  solution in an augmented plane wave method: application to copper}, Journal
  of Physics F: Metal Physics, 5 (1975), p.~2041.

\bibitem{kozlov1997elliptic}
{\sc V.~A. Kozlov, V.~Mazia, and J.~Rossmann}, {\em Elliptic boundary value
  problems in domains with point singularities}, vol.~52, American Mathematical
  Soc., 1997.

\bibitem{kresse1996efficient}
{\sc G.~Kresse and J.~Furthm\"uller}, {\em Efficient iterative schemes for
  \textit{ab initio} total-energy calculations using a plane-wave basis set},
  Phys. Rev. B, 54 (1996), pp.~11169--11186.

\bibitem{kresse99}
{\sc G.~Kresse and D.~Joubert}, {\em From ultrasoft pseudopotentials to the
  projector augmented-wave method}, Phys. Rev. B, 59 (1999), pp.~1758--1775.

\bibitem{kuchment2012floquet}
{\sc P.~A. Kuchment}, {\em Floquet theory for partial differential equations},
  vol.~60, Birkh{\"a}user, 2012.

\bibitem{levitt2015parallel}
{\sc A.~Levitt and M.~Torrent}, {\em Parallel eigensolvers in plane-wave
  {D}ensity {F}unctional {T}heory}, Computer Physics Communications, 187
  (2015), pp.~98--105.

\bibitem{li2019discontinuous}
{\sc X.~Li and H.~Chen}, {\em A discontinuous {G}alerkin scheme for
  full-potential electronic structure calculations}, J. Comput. Phys., 385
  (2019), pp.~33--50.

\bibitem{maday2019analyticity}
{\sc Y.~Maday and C.~Marcati}, {\em Analyticity and hp discontinuous {G}alerkin
  approximation of nonlinear {S}chrödinger eigenproblems}, arXiv preprint
  arXiv:1912.07483,  (2019).

\bibitem{maday2019hpdiscontinuous}
{\sc Y.~Maday and C.~Marcati}, {\em Regularity and {$hp$} discontinuous
  {G}alerkin finite element approximation of linear elliptic eigenvalue
  problems with singular potentials}, Math. Models Methods Appl. Sci., 29
  (2019), pp.~1585--1617.

\bibitem{melrose93}
{\sc R.~B. Melrose}, {\em The {A}tiyah-{P}atodi-{S}inger index theorem}, vol.~4
  of Research Notes in Mathematics, A K Peters, Ltd., Wellesley, MA, 1993.

\bibitem{reed1978iv}
{\sc M.~Reed and B.~Simon}, {\em Methods of modern mathematical physics. {IV}.
  {A}nalysis of operators}, Academic Press [Harcourt Brace Jovanovich,
  Publishers], New York-London, 1978.

\bibitem{slater1937wave}
{\sc J.~C. Slater}, {\em Wave functions in a periodic potential}, Physical
  Review, 51 (1937), p.~846.

\bibitem{Solovej1991}
{\sc J.~P. Solovej}, {\em Proof of the ionization conjecture in a reduced
  {H}artree-{F}ock model}, Invent. Math., 104 (1991), pp.~291--311.

\bibitem{torrent2008337}
{\sc M.~Torrent, F.~Jollet, F.~Bottin, G.~Zérah;, and X.~Gonze}, {\em
  Implementation of the projector augmented-wave method in the {ABINIT} code:
  {A}pplication to the study of iron under pressure}, Computational Materials
  Science, 42 (2008), pp.~337 -- 351.

\bibitem{weinberger1974variational}
{\sc H.~F. Weinberger}, {\em Variational methods for eigenvalue approximation},
  SIAM, Philadelphia, Pa., 1974.

\end{thebibliography}

\end{document}